\pdfoutput=1

\documentclass[reqno,a4paper,pdftex,dvipsnames]{amsart}
\usepackage{microtype}
\usepackage{color}
\usepackage{amssymb}
\usepackage{stmaryrd}
\usepackage[utf8]{inputenc}
\usepackage{a4wide}
\usepackage{paralist}
\usepackage{enumitem}
\usepackage{ifthen}
\usepackage{graphicx}
\usepackage{slashed}
\usepackage{multicol}

\usepackage{amsmath}
\usepackage{amsthm}
\usepackage{amsfonts}
\usepackage{sseq}
\usepackage{array}
\usepackage[all]{xy}
\definecolor{darkblue}{rgb}{0,0,0.9}
\definecolor{lightblue}{rgb}{.5,.5,.9}
\usepackage[pdftex,breaklinks,colorlinks,filecolor=blue,linkcolor=blue,citecolor=blue,urlcolor=blue,hypertexnames=true,plainpages=false]{hyperref}

\usepackage{tikz}
\usepackage{etoolbox}

\usepackage[margin=3in]{geometry}                           
\usepackage{mathtools}
\usepackage{dsfont}
\usepackage{tikz-cd}

\DeclareMathOperator{\TO}{TO}
\DeclareMathOperator{\To}{to}

\DeclareMathOperator{\Sp}{sp}

\newcommand{\etp}{\, \wh \otimes \,}

\newcommand{\idbb}{\mathds{1}}                              
\newcommand{\Omegapmone}{\Omega _{\pm \idbb}SO_{2k}}       
           
\newcommand{\Omegamin}{\Omega_{\pm \idbb}^{\mathrm{min}} SO_{2k}}

\newcommand{\TB}{\operatorname{TB}}

\newcommand{\FSOcrossSO}{\text{Fr}(E) \times _{SO_{2k}} SO_{2k}}

\newcommand{\Fr}{\operatorname{Fr}}
                             
\newcommand{\Turn}{\mathrm{Turn}}   
\newcommand{\Map}{\mathrm{Map}}                
\newcommand{\Turnmin}{\Turn^{\mathrm{min}}}

\newcommand{\ad}{\mathrm{ad}}

\newcommand{\GG}{\mathcal{G}}
\newcommand{\cc}{\mathrm{cc}}

\makeatletter
\providecommand\@dotsep{5}
\def\listtodoname{List of Todos}
\def\listoftodosfix{\@starttoc{tdo}\listtodoname}
\makeatother

\usepackage{xargs}                                          
\usepackage[]{xcolor}                                       
\usepackage[colorinlistoftodos,prependcaption,textsize=tiny]{todonotes}
\newcommandx{\unsure}[2][1=]{\todo[linecolor=red,backgroundcolor=red!25,bordercolor=red,#1]{#2}}
\newcommandx{\change}[2][1=]{\todo[linecolor=blue,backgroundcolor=blue!25,bordercolor=blue,#1]{#2}}
\newcommandx{\info}[2][1=]{\todo[linecolor=OliveGreen,backgroundcolor=OliveGreen!25,bordercolor=OliveGreen,#1]{#2}}
\newcommandx{\improvement}[2][1=]{\todo[linecolor=Plum,backgroundcolor=Plum!25,bordercolor=Plum,#1]{#2}}
\newcommandx{\thiswillnotshow}[2][1=]{\todo[disable,#1]{#2}}

\numberwithin{equation}{section}

\newcommand\C{\mathbb{C}}

\newcommand\R{\mathbb{R}}

\newcommand\Z{\mathbb{Z}}

\newcommand\del{\partial}

\newcommand{\Aut}{\textup{Aut}}

\newcommand{\id}{\textup{Id}}
\newcommand{\pt}{\textup{pt}}

\newcommand{\an}[1]{\langle{#1}\rangle}
\newcommand{\An}[1]{\left\langle{#1}\right\rangle}
\newcommand{\wt}{\widetilde}
\newcommand{\wh}{\widehat}
\newcommand{\ol}{\overline}
\newcommand{\ul}{\underline}
\newcommand{\xra}{\xrightarrow}

\theoremstyle{plain}
\newtheorem{theorem}{Theorem}[section]
\newtheorem{lemma}[theorem]{Lemma}
\newtheorem{corollary}[theorem]{Corollary}
\newtheorem{proposition}[theorem]{Proposition}

\theoremstyle{definition}
\newtheorem{definition}[theorem]{Definition}
\newtheorem{example}[theorem]{Example}

\theoremstyle{remark}
\newtheorem{remark}[theorem]{Remark}
\newtheorem{question}[theorem]{Question}

\usepackage{color}

\advance \topmargin by -\headheight
\advance \topmargin by -\headsep
     
\evensidemargin \oddsidemargin
\title{Turning vector bundles}
\author{Diarmuid Crowley, Csaba Nagy, Blake Sims and Huijun Yang}
\date{\today}

\begin{document}
\maketitle

\begin{abstract}
We define a turning of a rank-$2k$ vector bundle $E \to B$ to be a homotopy
of bundle automorphisms  $\psi_t$ from $\idbb_E$, the identity of $E$, 
to  $-\idbb_E$, minus the identity, 
and call a pair $(E, \psi_t)$ a turned bundle.
We investigate when vector bundles admit turnings
and develop 
the theory of turnings and their obstructions.
In particular, we determine which rank-$2k$ bundles over the $2k$-sphere
are turnable.

If a bundle is turnable, then it is orientable.
In the other direction, 
complex bundles are turned bundles and  
for bundles over finite $CW$-complexes with rank in the stable range, 
Bott's proof of his periodicity theorem shows that
a turning of $E$ defines a homotopy class of complex structure on $E$.
On the other hand, we give examples of rank-$2k$ bundles over $2k$-dimensional spaces,
including the tangent bundles of some $2k$-manifolds, which are turnable
but do not admit a complex structure.  Hence turned bundles can be viewed
as generalisations of complex bundles.

We also generalise the definition of turning to other settings, including other paths of automorphisms, 
and we relate the generalised turnability of vector bundles to the
topology of their gauge groups and the computation of certain Samelson products.

\end{abstract}

\section{Introduction}  \label{s:intro}
Let $\pi \colon E \to B$ be a real Euclidean vector bundle over a base space $B$, 
which for simplicity we assume is connected.
The bundle $E$ has two canonical automorphisms: $\idbb_E$, the identity
and $-\idbb_E$, the automorphism which takes a vector to its negative.
A {\em turning} of $E$ is a continuous path $\psi_t$ of bundle automorphisms from $\idbb_E$ to 
$-\idbb_E$: if a turning of $E$ exists, we call $E$ {\em turnable} and the pair $(E, \psi_t)$
a {\em turned} vector bundle.
The {\em turning problem} for $E$ is 
to determine whether $E$ is turnable.

While the turning problem is a natural topological problem
amenable to classical methods in algebraic topology, to the best of our knowledge it 
has not been explicitly discussed in the literature.  
Our primary interest in turnings stems from the fact that they generalise complex structures.
As we explain in Section~\ref{s:stable}, Bott's original proof of his periodicity theorem, 
shows for bundles over finite $CW$-complexes that stable turnings are equivalent 
to stable complex structures.  On the other hand, we discovered that there are unstable bundles which are turnable
but do not admit a complex structure; e.g.\ see the discussion following Theorem~\ref{t:spheres}.   
Hence there is a sequence of strict inclusions:
\[ \{Complex~bundles\} \subsetneq \{Turned~bundles\} \subsetneq \{Stably~complex~bundles\} \]

The turning problem and its generalisations also arise naturally in the study of gauge groups.
Turnings are closely related to the group of components and the fundamental group
of the gauge group associated to $E$, and by studying turnings we can gain information about the 
low-dimensional topology of gauge groups; see e.g.\ Theorem~\ref{t:eta-intro}. 
The generalised turning problem for loops in the structure group is also related certain
Samelson products and our results on turnings lead to new computations of Samelson
products, which have implications for the high-dimensional topology of
certain gauge groups; see Proposition~\ref{p:GGE_eta}.

\subsection{Turnability of vector bundles}
We begin with some elementary remarks on the turning problem. 
Given $b \in B$, let $E_b := \pi^{-1}(b)$ be the fiber over $b$, which is a Euclidean vector 
space.
Since a turning of $E$ restricts to a path from $\idbb_{E_b}$ to $-\idbb_{E_b}$,
if $E$ is turnable then the rank of $E$ must be even.
Moreover, a turning of $E$ can be used to continuously orient each fibre $E_b$.
Hence if $E$ is turnable, then $E$ is orientable; see the discussion prior to Lemma \ref{l:tble-oble}.
On the other hand, if $E$ admits a complex structure then the map $t \mapsto e^{i\pi t}\idbb_E$
defines a turning of $E$ and so complex bundles are turned bundles.
Since oriented rank-$2$ bundles are equivalent to complex line bundles, they are turnable
and so we assume $k > 1$, unless stated otherwise.

We next discuss the turning problem stably.
Suppose that the base space $B$ is (homotopy equivalent to) a finite $CW$-complex, 
let $\ul \R^j$ denote the trivial rank-$j$ bundle over $B$ and let $E \oplus \ul \R^j$ denote the Whitney sum of $E$ and $\ul \R^j$.
We say that $E$ is {\em stably turnable} if $E \oplus \ul \R^j$ is turnable for some $j \geq 0$ 
and similarly we say that $E$ admits a stable complex structure if 
$E \oplus \ul \R^j$ admits a complex structure for some $j \geq 0$.
Then we have the following result (see Theorem \ref{t:tc} for a more refined statement).

\begin{theorem} \label{t:1}
Let $E \to B$ be a vector bundle over a finite-dimensional
$CW$-complex.  Then $E$ is stably turnable if and only if 
$E$ admits a stable complex structure.
\end{theorem}

The question of whether $E$ admits a stable complex structure, while in general
difficult, can be characterised entirely using the ring structure in real $K$-theory;
see Proposition~\ref{p:eta}.  
We therefore turn our attention to the turning problem for bundles outside the stable range
and in this paper we pay close attention to rank-$2k$ bundles over $2k$-dimensional $CW$-complexes.  
Such bundles are just outside the stable range
and provide a large class of interesting examples, including the tangent bundles $TM$
of orientable smooth $2k$-manifolds $M$ and rank-$2k$ bundles over $S^{2k}$.
Starting with $TS^{2k}$,
we recall that Kirchoff~\cite{k47} proved that if $TS^{2k}$ admits a complex structure then $TS^{2k+1}$ is trivial.
Shortly afterwards Borel and Serre~\cite{b-s51} showed that $TS^{2k}$ admits a complex structure 
if and only if $k = 0, 1$ or $3$ and a little later Bott and Milnor~\cite{bm58} showed that 
$TS^{2k+1}$ is trivial if and only if $k = 0, 1$ or $3$.
An important first result on the turning problem is
the following strengthening of Kirchoff's theorem (see Theorem \ref{t:SKirchoff}).

\begin{theorem} \label{t:2}
If $TS^{2k}$ is turnable then $TS^{2k+1}$ is trivial.
\end{theorem}

Combined with the results of Borel and Serre and Bott and Milnor, 
Theorem~\ref{t:2} shows that $TS^{2k}$ is turnable if and only if it admits a complex structure.
However, such a statement does not hold generally, even for rank-$2k$ bundles $E$ over $S^{2k}$
as Theorem~\ref{t:spheres} below shows.

We next consider the turnability of all oriented rank-$2k$ bundles over $S^{2k}$.
Stable vector bundles over $S^{2k}$ are classified by the real $K$-theory groups
$\wt{KO}(S^{2k})$ which are, respectively, isomorphic to $\Z$, $\Z/2$ and $0$ for $k$ respectively even,
congruent to $1$ mod~$4$ or congruent to $3$ mod~$4$.
Given an oriented rank-$2k$ vector bundle $E \to S^{2k}$, we let $\xi_E \in \wt{KO}(S^{2k})$
denote the reduced real $K$-theory class defined by $E$.
When $k = 2$, oriented vector bundles over $S^4$ admit unique homotopy classes of spin structures.
By \cite{k59}, the spin characteristic class $p = \frac{p_1}{2}$ defines an isomorphism
$p \colon \wt{KO}(S^4) \to H^4(S^4; \Z)$ and by \cite{w62}, $\rho_2(p(\xi_E)) = \rho_2(\mathrm{e}(E))$
for all oriented rank-$4$ bundles $E \to S^4$, where $\rho_d$ denotes reduction mod~$d$
and for any base space $B$, $\mathrm{e}(E) \in H^{2k}(B; \Z)$ denotes the Euler class of an oriented rank-$2k$ 
bundle $E \to B$. As a corollary of Theorem \ref{t:to_S2k} we obtain

\begin{theorem} \label{t:spheres}
For $k > 1$, let $E \to S^{2k}$ be an oriented rank-$2k$ bundle.
Then $E$ is turnable if and only if one of the following holds:
\begin{compactenum}[a)]
\item $k = 2$ and $\rho_4(\mathrm{e}(E) + p(\xi_E)) = 0$ or $\rho_4(\mathrm{e}(E) - p(\xi_E)) = 0$;
\item $k = 3$;
\item $k > 2$ is even, $\rho_4(\mathrm{e}(E)) = 0$ and $\rho_2(\xi_E) = 0$;
\item $k > 3$ is odd and $\rho_4(\mathrm{e}(E)) = 0$.
\end{compactenum}
\end{theorem}

Theorem~\ref{t:spheres} gives many examples of bundles which are turnable but do not admit a complex structure. 
For example, for $m \geq 2$ let $\tau_{m} \in \pi_{m-1}(SO_{m})$ denote the homotopy class of the clutching function of $TS^{m}$
and for $n \in \Z$ let $nTS^m$ denote the bundle corresponding to $n\tau_{m} \in \pi_{m-1}(SO_{m})$. 
Then for $m = 4j$, $nTS^{4j}$ is turnable if and only if $n$ is even, whereas 
by a theorem of Thomas~\cite[Theorem 1.7]{t67}, $nTS^{4j}$ admits a complex structure if and only if $n = 0$.

Theorem~\ref{t:spheres} leads to a general result on the turnability of
rank-$2k$ bundles over general finite $2k$-dimensional $CW$-complexes.
If $J$ is a complex structure on $E \oplus \ul \R^{2j}$ for some $j \geq 0$,
let $c_k(J) \in H^{2k}(B; \Z)$ denote the $k^{\text{th}}$ Chern class of $J$.
We define the subgroup $I^{2k}(B) \subseteq H^{2k}(B; \Z/4)$ by
\[ I^{2k}(B) : = 
\begin{cases}
((\times 2) \circ Sq^2 \circ \rho_2)(H^{2k-2}(B; \Z)) & \text{$k$ is odd}, \\
0 & \text{$k$ is even},
\end{cases}
\]
where $Sq^2$ is the second Steenrod square and $\times 2$ is the natural map induced by the inclusion of coefficients $\times 2 \colon \Z/2 \to \Z/4$.  The following result is a simple consequence of Theorems \ref{t:TC1} and \ref{t:TC2}.

\begin{theorem} \label{t:4}
Let $E \to B$ be an oriented rank-$2k$ vector bundle,
over a finite $CW$-complex of dimension at most $2k$
and if $k$ is even, assume that $H^{2k}(B; \Z)$ contains no $2$-torsion.
Then $E$ is turnable if and only if there is a $j \geq 0$ and a complex structure $J$ on $E \oplus \ul \R^{2j}$ such that
\[ [\rho_4(c_k(J))] = \pm [\rho_4(\mathrm{e}(E))] \in H^{2k}(B; \Z/4)/I^{2k}(B).\] 
\end{theorem}

\begin{remark} \label{r:1}
\noindent
We point out that when
$k$ is even and $H^{2k}(B; \Z)$ contains $2$-torsion, the condition in Theorem~\ref{t:4} remains necessary
but is no longer sufficient; see Theorem~\ref{t:TC1} and Example~\ref{e:not_suff}.
\end{remark}

Theorem~\ref{t:4} shows that there are many examples of
manifolds whose tangent bundles are turnable but do not admit a complex structure.  
For example, 
if $M_l = \sharp_l(S^4 \times S^4)$ is the $l$-fold connected sum of $S^4 \times S^4$ with itself,
then for any $j > 0$, the bundle $TM_l \oplus \ul \R^{2j}$ admits two homotopy classes of complex
structure $J$, each with $c_4(J) = 0$.  But $\mathrm{e}(TM_l) = \pm 2(l+1)$ by the 
Poincar\'e-Hopf Theorem; see \cite[p.\ 113]{h89}.
It follows that $TM_l$ does not admit a complex structure and that 
$TM_l$ is turnable if and only if $l$ is odd; 
e.g.\ $T(S^4 \times S^4)$ is turnable but does not admit a complex structure.  
For a more general statement about when $TM$ is 
turnable but does not admit a complex structure, see Corollary~\ref{c:TC2}.

\subsection{The turning obstruction for bundles over suspensions}
In order to study the turning problem and obtain most of our results above, we define a complete obstruction to the existence of turnings for bundles over suspensions. For this we need to refine our definition of turning by specifying the homotopy class of the turning in a fibre. If $\psi_t$ is a turning of an oriented rank-$2k$ bundle $E \to B$, then by restricting $\psi_t$ to a fibre $E_b$ 
we obtain a path of isometries from $\idbb_{E_b}$ to $-\idbb_{E_b}$.
When $E_b$ is identified with $\R^{2k}$ via an orientation-preserving isomorphism,
this path is identified with a path in $SO_{2k}$ from $\idbb$ to $-\idbb$,
which is well-defined up to path homotopy. 
Given a path $\gamma$ in $SO_{2k}$ from $\idbb$ to $-\idbb$, we shall call a turning a {\em  $\gamma$-turning} if its restriction to each fibre is path homotopic to $\gamma$.  
When $k > 1$, as we generally assume, $\pi_1(SO_{2k}) \cong \Z/2$, so there are precisely two path homotopy classes of possible paths. 

Suppose that the base space $B = SX$ is a suspension, i.e.\ it is the union of two copies of the cone on $X$.  
The restriction of $E$ to each cone admits a $\gamma$-turning, which is unique up 
to homotopy and $E$ admits a $\gamma$-turning if and only if the 
$\gamma$-turnings over the cones agree over $X$ up to homotopy.  
We can then define the \textit{$\gamma$-turning obstruction of $E$} by measuring the difference of 
the $\gamma$-turnings over the cones and there are several equivalent ways to do this, which we present in Section \ref{ss:tovb}.
Here we discuss what we later call the {\em adjointed $\gamma$-turning obstruction}.
Recall that the set of isomorphism classes of oriented rank-$2k$-bundles over $SX$ form a group, which is naturally isomorphic to $[X, SO_{2k}]$ via the map which sends the isomorphism class of $E$ to the homotopy class of its clutching function $g \colon X \to SO_{2k}$. We define the adjointed $\gamma$-turning obstruction
\begin{equation} \label{eq:TO-intro}
 \TO_\gamma \colon [X, SO_{2k}] \to [SX, SO_{2k}],
\qquad [g] \mapsto \left[ [x, t] \mapsto g(x)\gamma(t)g(x)^{-1} \right],
\end{equation} 
where $[x, t] \in SX$ is the point defined by $(x, t) \in X \times I$.  
The following result, which follows from Proposition \ref{p:to} and Lemma \ref{l:to+bto}, justifies calling
$\TO_\gamma$ the $\gamma$-turning obstruction.

\begin{proposition} \label{p:TO-intro}
Let $E \to SX$ be an oriented rank-$2k$ vector bundle with clutching function $g \colon X \to SO_{2k}$.
Then $E$ is $\gamma$-turnable if and only if $\TO_\gamma([g]) = 0$.
Moreover, if $X$ is itself a suspension, then $\TO_\gamma$ is a homomorphism of abelian groups.
\end{proposition}

Proposition~\ref{p:TO-intro} states that 
the $\gamma$-turning obstruction is additive for bundles over double suspensions.
This 
is an essential input to the Theorem~\ref{t:to_S2k},
which largely computes $\TO_\gamma$ for rank-$2k$ bundles over the $2k$-sphere
and both homotopy classes of paths $\gamma$.
Theorem \ref{t:spheres} above is an immediate corollary of Theorem~\ref{t:to_S2k}.
 
The final element in the proof of Theorem~\ref{t:to_S2k} involves 
generalising the turning problem.
The definition of the $\gamma$-turning obstruction naturally leads us to consider
the turning obstruction for an essential loop $\eta \colon I \to SO_{2k}$ with $\eta(0) = \eta(1) = \idbb$.
If we replace $\gamma(t)$ by $\eta(t)$ in \eqref{eq:TO-intro} above, 
we obtain the function
\[ \TO_\eta \colon [X, SO_{2k}] \to [SX, SO_{2k}], \qquad [g] \mapsto \left[ [x, t] \mapsto g(x)\eta(t)g(x)^{-1} \right].\]
If $E \to SX$ is a bundle with clutching function $g : X \to SO_{2k}$, then $\TO_\eta([g])$ is a complete obstruction to finding a loop $\psi_t$ of bundle automorphisms of $E$ based at the identity such that the restriction of $\psi_t$ to a fibre $E_b$ is an
essential loop of isometries of $E_b$.
Moreover, if $\gamma$ is a path in $SO_{2k}$ from $\idbb$ to $-\idbb$, then the concatenation of paths $\eta \ast \gamma$ 
represents the other path homotopy class of such paths. 
Hence a bundle $E \to SX$ with clutching function $g$
is turnable if and only if one of $\TO_\gamma([g])$ or $\TO_{\eta \ast \gamma}([g])$ vanishes.
The following result
relates $\TO_\eta$ and $\TO_\gamma$ and states that $\TO_\gamma$ is in general $4$-torsion;
see also Theorem \ref{thm:gpd-appl}.

\begin{theorem} \label{t:to-ord24}
Let $E$ be an oriented rank-$2k$ vector bundle with clutching function $g \colon X \to SO_{2k}$. Then
\begin{compactenum}[a)]
\item $2\TO_\eta([g]) = 0$;
\item $\TO_{\eta \ast \gamma}([g]) = \TO_\eta([g]) + \TO_{\gamma}([g])$;
\item If $k$ is even, then $2\TO_{\gamma}([g]) = 0$;
\item If $k$ is odd, then $2\TO_{\gamma}([g]) = \TO_\eta([g])$ and $4\TO_\gamma([g]) = 0$.
\end{compactenum}
\end{theorem}

\begin{remark}
Notwithstanding Theorem~\ref{t:to-ord24}(d), we know of no example of a bundle $E \to SX$
with clutching function $g$, where $2 \TO_\gamma([g]) \neq 0$.
In particular, by Theorem \ref{t:to_S2k}, $2 \TO_\gamma(\tau_{4k+2}) = \TO_\eta([\tau_{4k+2}])  = 0$ for all $k \geq 1$.
The proof we give of this result is computational and somewhat surprising to us.
It would be interesting to know if there is a space $X$ and a clutching function $g \colon X \to SO_{2k}$ 
with $2 \TO_\gamma([g]) \neq 0$.
\end{remark}

\subsection{General turnings, the topology of gauge groups and Samelson products}
Let $\Fr(E)$ denote the frame bundle of an oriented vector bundle $E \to B$,
which is a principal $SO_{2k}$-bundle over $B$.
The group of automorphisms of $E$ is canonically homeomorphic to the gauge
group of $\Fr(E)$, and so the turning problem can be viewed as a problem
in the topology of gauge groups:
we are asking when a topological feature of the structure group extends to the whole gauge group.

To describe general turning problems, we let $G$ be a path-connected topological group, 
for example a connected Lie group, and $P \to B$ be a principal $G$-bundle with gauge group $\GG_P$:
if $G = SO_{2k}$ and $P = \Fr(E)$, then we shall write $\GG_E$ in place of $\GG_{\Fr(E)}$.
If $Z(G)$ denotes the centre of $G$,
then multiplication by $z \in Z(G)$ defines an element $z_P \in \GG_P$.  
Given a path $\gamma \colon I \to G$ between elements of $Z(G)$,
the {\em $\gamma$-turning problem} for $P$ is to determine whether there is a path 
$\psi_t$ in $\GG_P$ with $\psi(0) = \gamma(0)_P$ and $\psi(1) = \gamma(1)_P$
and whose restriction to a fibre is path homotopic to $\gamma$.

When $B = SX$ is a suspension, then principal $G$-bundles $P \to SX$
are determined up to isomorphism by their clutching functions $g \colon X \to G$
and the definition and properties of the $\gamma$-turning obstruction 
for vector bundles generalise in the obvious way.
The (adjointed) $\gamma$-turning obstruction is the map 
\[ \TO_\gamma \colon  [X, G] \to [SX, G], \qquad  [g] \mapsto 
\left[ [x, t] \mapsto g(x)\gamma(t)g(x)^{-1} \right], \]
and $P$ is $\gamma$-turnable if and only if $\TO_\gamma([g]) = 0$;
see Remark~\ref{r:gato}.
If we allow $\gamma$ to vary among all paths between central elements of $G$,
the path homotopy classes of the possible paths $\gamma$ form a groupoid, 
which is a full subcategory of the fundamental groupoid of $G$.
We call this groupoid the \textit{central groupoid of $G$} and denote it by $\pi^Z(G)$.  
If we fix $[g] \in [X, G]$, 
then we can regard $\TO_{\gamma}([g])$ as a function of $\gamma$. The resulting map 
\[ \pi^Z(G) \rightarrow [SX,G] \]
is a morphism of groupoids (where the group $[SX,G]$ is regarded as a groupoid on one object)
and this general point of view allows us to prove Theorem~\ref{t:to-ord24}.

Returning to vector bundles $E$ and the topology of their gauge groups $\GG_E$,
the $\eta$-turning problem which has the most direct implications (see Theorem \ref{t:gauge-eta}): 

\begin{theorem} \label{t:eta-intro}
If $B = SX$ is a suspension and $n_{SX} := \left| [SX, SO_{2k}] \right|$ is finite,
then for any rank-$2k$ vector bundle $E \to SX$
\[ 
\bigl| \pi_0(\GG_E) \bigr| = 
\begin{cases}
n_{SX} & \text{if $E$ is $\eta$-turnable,} \\
\frac{n_{SX}}{2} & \text{if $E$ is not $\eta$-turnable.}
\end{cases}\]
\end{theorem}
\noindent
Theorem~\ref{t:eta-intro} shows that when $[SX, SO_{2k}]$ is finite, for example
when $SX = S^{2k}$, then the $\eta$-turnability of
a vector bundle $E$ is a homotopy invariant of its gauge group $\GG_E$.
While the turnability of $E$ is not {\em a priori} a homotopy invariant of $\GG_E$, 
recent work of Kishimoto, Membrillo-Solis and Theriault \cite{t21} on the homotopy 
classification of the gauge groups of rank-$4$ bundles $E \to S^4$,
when combined with our results in Theorem \ref{t:to_S2k}, does show that the
turnability of these bundles is a homotopy invariant of their gauge groups:
See Proposition \ref{p:tt} for a more detailed statement.

We compute $\TO_\eta$ 
for all rank-$2k$ bundles over $S^{2k}$ in Theorem~\ref{t:to_S2k}.  
In fact in this case $\TO_\eta([g]) = \an{[g], \eta}$ 
is the Samelson product
of $[g] \in \pi_{2k-1}(SO_{2k})$ and $\eta \in \pi_1(SO_{2k})$; 
see Lemma \ref{l:TO_and_SP}.
On the other hand, 
Samelson products are in general delicate to calculate
and so the computations of $\TO_\eta([g])$ in Theorem~\ref{t:to_S2k}, which are carried out using the point of view
of the turning obstruction, may be of independent interest.  For example, (see Proposition \ref{p:sam-eta}),
for $\eta_{4j-1} \colon S^{4j} \to S^{4j-1}$ an essential map, we have 

\begin{corollary} \label{c:eta-intro}
The Samelson product $\an{\tau_{2k}, \eta}$ 
is given by $\an{\tau_{4j+2}, \eta} = 0$ and $\an{\tau_{4j}, \eta} = \tau_{4j}\eta_{4j-1} \neq 0$.
\end{corollary}
\noindent
Corollary~\ref{c:eta-intro} also has implications for the high-dimensional homotopy groups of certain gauge groups;
see Proposition~\ref{p:GGE_eta} in Section~\ref{ss:SP_and_TO}.

\subsection*{Organisation}  
The rest of this paper is organised as follows.
In Section~\ref{s:prelim} we set up the necessary preliminaries to discuss the turning problem.
We define turnings and $\gamma$-turnings, universal bundles which classify turnings 
and relate the turning problem to the topology of the gauge group.
In Section~\ref{s:to} we define the $\gamma$-turning obstruction for bundles
over suspensions and develop the theory of the $\gamma$-turning obstruction,
regarded as a map from the central groupoid of a path-connected topological group $G$.
We also show that $\eta$-turning obstruction is given by certain Samelson products.
In Section~\ref{s:2k_over_S2k}, we consider rank-$2k$ vector bundles over the $2k$-sphere
and compute their turning obstructions in detail, proving Theorem~\ref{t:spheres}.
In Section \ref{s:stable}, we consider the turning problem for bundles in the stable range
and prove Theorem~\ref{t:1}.
Finally, in Section~\ref{s:2k_over_2k} we combine the results of Sections~\ref{s:2k_over_S2k} 
and~\ref{s:stable} on rank-$2k$ vector 
bundles over the $2k$-sphere and stable vector bundles to prove
Theorem~\ref{t:4} on rank-$2k$ vector bundles over $2k$-dimensional $CW$-complexes.

\subsection*{Background information}
This paper is based on the MSc thesis of the third author, also entitled ``Turning vector bundles'',
which was submitted to the University of Melbourne in May 2020 under 
the supervision of the first and second authors.

\subsection*{Acknowledgements}
The first and fourth authors gratefully acknowledge the support of the China Scholarship Council, 
sponsor for the visit of the fourth author to the University of Melbourne in 2018.
The second author was supported by EPSRC New Investigator grant EP/T028335/1.

\section{Turnings and gauge groups}  \label{s:prelim}
In this section we set up the necessary definitions and notation for the turning problem and establish some basic results. 
In Section \ref{ss:turnings} we define turnings and $\gamma$-turnings and introduce the terminology to describe the relationship between turnings and orientations of a vector bundle. A more general notion of turning, for principal $G$-bundles, is defined in Section \ref{ss:gt}. 
In Section \ref{ss:assoc-b} we define the associated turning bundle of a vector bundle and construct a universal turned bundle. 
We also establish some equivalent conditions to turnability in terms of the associated turning bundle and the universal turned bundle. 
In Sections \ref{ss:GG} and \ref{ss:gg-eta} we study the connection between the turnability of a vector bundle and the low-dimensional homotopy groups of its gauge group.

\subsection{Turnings} \label{ss:turnings}
All vector spaces $V$ in this paper are real and Euclidean.  The connected component of
the group of isometries of $V$ is denoted $SO(V)$, $\idbb \in SO(V)$ is the identity
and $-\idbb \in SO(V)$ is defined by $-\idbb(v) = -v$ for all $v \in V$.
We use $\R^j$ to denote $j$-dimensional
Euclidean space with its standard metric and as usual we set $SO_j := SO(\R^j)$.

All vector bundles $\pi \colon E \to B$ are real and Euclidean and for simplicity
we assume that the base space $B$ is connected.   We denote the trivial bundle $\R^j \times B \to B$
by $\ul \R^j$; the base space will either be specified or clear from the context.
We shall use the symbol $E$ to ambiguously
denote both the total space of the bundle and the bundle itself.  For $b \in B$, $E_b : = \pi^{-1}(b)$ 
is the fiber of $E$ over $b$, which is a vector space.
Let $I := [0, 1]$ be the unit interval.

\begin{definition}
Let $V$ be an even-dimensional real vector space, so that $-\idbb \in SO(V)$. A {\em turning} of $V$ is a path 
$\gamma \colon I \rightarrow SO(V)$ from $\idbb$ to $-\idbb$.
\end{definition}

\noindent
In particular, a turning of $\R^{2k}$ is a path in $SO_{2k}$ from $\idbb$ to $-\idbb$
and we write $\Omegapmone$ for the mapping space $\text{Map} \bigl( (I, (\{0\}, \{1\})), (SO_{2k} , (\{\idbb\}, \{-\idbb\})) \bigr)$ consisting of all turnings of $\R^{2k}$, with the compact-open topology. 
Note that $\Omega SO_{2k}$, the spaces of loops based at $\idbb$, acts freely and transitively on
$\Omegapmone$ by point-wise multiplication.  Hence choosing $\gamma \in \Omegapmone$ 
defines a homeomorphism from
$\Omegapmone$ to $\Omega SO_{2k}$
and we will use this homeomorphism to compute the
homotopy groups of $\Omegapmone$.

\begin{definition}[Standard turning of $\R^{2k}$] \label{d:stdT}
Let $\R^{2k} = \C^k$ define the standard complex structure on $\R^{2k}$ and let 
$U_k \subseteq SO_{2k}$ be the unitary subgroup.
The {\em standard turning} of $\R^{2k}$ is the path
\[
\beta : I \rightarrow SO_{2k}, \quad t \mapsto e^{i \pi t} \idbb \in U_k \subseteq SO_{2k} \text{\,.}
\]
\end{definition}

If $2k > 2$, then $\pi_0(\Omegapmone) \cong \pi_1(SO_{2k}) \cong \Z/2$, so there are two turnings of $\R^{2k}$ up to homotopy.
Indeed, if $\bar \beta$ is a representative of the other homotopy class
and
\[ \eta \colon (I, \{0, 1\}) \to (SO_{2k}, \idbb) \]
is a loop representing the generator of $\pi_1(SO_{2k})$,
then $[\bar \beta] = [\eta \ast \beta]$, where $\ast$ denotes concatenation of paths and $[\gamma]$ denotes
the path homotopy class of a path $\gamma$.
If we pointwise conjugate $\beta$ with a fixed element of 
$O_{2k} \setminus SO_{2k}$, then we obtain a path in $[\bar \beta]$;
equivalently, an orientation-reversing isomorphism $\R^{2k} \rightarrow \R^{2k}$ pulls back $\beta$ to a 
turning that is path homotopic to $\bar \beta$. Note that the turning defined by the formula 
$t \mapsto e^{-i \pi t}\idbb$ is path homotopic to $\beta$ if $k$ is even and to $\bar \beta$ if $k$ is odd.

Let $V$ be a vector space of dimension $2k$ equipped with a turning and an orientation. If $2k > 2$, then 
we say that the turning and the orientation are compatible if the turning is homotopic to $\beta$ under an orientation-preserving identification $V \cong \R^{2k}$. If $2k = 2$, then they are compatible if the turning is homotopic to the path $t \mapsto e^{ri \pi t}\idbb$ for some positive (odd) $r$ under an orientation-preserving identification 
$V \cong \R^2$. In both cases there is a unique orientation of $V$ which is compatible with a given turning, hence we obtain a well-defined map from the homotopy classes of turnings of $V$ to its orientations. This map is a bijection if $2k > 2$ and surjective if $2k=2$.

\begin{definition}[Turning, turnable and turned] \label{d:Tbundle}
Let $\pi \colon E \to B$ be a rank-$2k$ vector bundle. A {\em turning} of $E$ is a path $\psi_t$ in the space of automorphisms of $E$ from $\idbb_E$ to $-\idbb_E$.  If a turning $\psi_t$ exists, we say that 
$E$ is {\em turnable} and a {\em turned} vector bundle is a pair $(E, \psi_t)$, where $\psi_t$ is a
turning of $E$.
\end{definition}

Clearly any trivial bundle is turnable and since bundle automorphisms can be pulled back along continuous maps, the pullback of a turnable bundle is turnable. Furthermore, every complex bundle $E$ is turnable via the path 
$t \mapsto e^{i \pi t}\idbb_E$.

A turning of a bundle restricts to a turning of each fibre, so by our earlier observations it determines an orientation on each fibre. Therefore we have

\begin{lemma} \label{l:tble-oble}
Every turnable bundle is orientable.  \qed
\end{lemma}

\noindent
If a rank-$2$ bundle is orientable, then it admits a complex structure, so we have 

\begin{proposition} \label{p:r2}
A rank-$2$ bundle is turnable if and only if it is orientable.
\qed
\end{proposition}

From now on we will focus on oriented bundles $\pi \colon E \to B$ of rank-$2k$ and we assume that $2k > 2$ 
unless otherwise stated.  
Since we are assuming that $B$ is connected, it follows that an orientable bundle has precisely two possible orientations and we let $\Bar{E}$ denote the same bundle with opposite orientation.

\begin{definition} \label{d:omg_gmm}
For a path $\gamma \in \Omegapmone$, let $\Omega_\gamma SO_{2k} \subset \Omegapmone$ denote the connected component of $\gamma$.
\end{definition}

\begin{definition}[$\gamma$-turnable, positive/negative turnable]
Let $E$ be a rank-$2k$ bundle and $\gamma \in \Omegapmone$. We say that $E$ is {\em $\gamma$-turnable}, if it has a turning whose restriction to each fibre $E_b$ lies in $\Omega_\gamma SO_{2k}$ under an orientation-preserving identification $E_b \cong \R^{2k}$. 

We will also say that $E$ is {\em positive/negative turnable}, if it has a turning which determines the positive/negative orientation on $E$.
\end{definition}

Obviously, if $\gamma$ is homotopic to $\gamma'$, then $\gamma$-turnability is equivalent to $\gamma'$-turnability. By definition positive turnability is equivalent to $\beta$-turnability and negative turnability is equivalent to $\bar{\beta}$-turnability. A bundle is turnable if and only if it is positive turnable or negative turnable. Finally, a bundle $E$ is positive turnable if and only if $\Bar{E}$ is negative turnable.
For the next definition, recall that $E$ is called {\em chiral} if $E$ is not isomorphic
to $\Bar{E}$.

\begin{definition}[Bi-turnable, strongly chiral]
If $E$ is both positive and negative turnable, we call it {\em bi-turnable}. If $E$ is turnable but not bi-turnable, we say that $E$ is {\em strongly chiral}.
\end{definition}

\noindent
If $E \cong \Bar{E}$, then $E$ cannot be strongly chiral. This shows that strong chirality implies chirality.

\begin{definition}[Turning type] \label{d:tability}
The {\em turning type} of an orientable rank-$2k$ bundle is the property of being either 
bi-turnable, strongly chiral or not turnable.
\end{definition}

\subsection{The associated turning bundle} \label{ss:assoc-b}
We can also think of a turning of a bundle as a continuous choice of turning in each fibre. 
To make this precise we define the associated turning bundle below, 
in analogy with the associated automorphism bundle.

Every oriented rank-$2k$ vector bundle $\pi \colon E \to B$ has an associated principal $SO_{2k}$-bundle, 
namely the frame bundle $\Fr(E)$, whose fibres consist of oriented, orthonormal frames of the fibres of $E$. 
We will view such frames as linear isomorphisms $\phi_b \colon \R^{2k} \to E_b$.
Then $SO_{2k}$ acts on the right on the total space of $\Fr(E)$ via pre-composition. 

\begin{definition}[The associated automorphism bundle and the associated turning bundle]
For an oriented, rank-$2k$ vector bundle $E\to B$ we define via the Borel construction:
\begin{compactenum}[a)]
\item The associated {\em automorphism bundle} 
\[
\Aut(E) \coloneqq \Fr(E) \times_{SO_{2k}}SO_{2k}  \rightarrow B,\]
where $SO_{2k}$ acts on itself by conjugation;
\item The associated {\em turning bundle}
\[
\Turn(E) \coloneqq \Fr(E) \times _{SO_{2k}} \Omegapmone \rightarrow B,\]
where $SO_{2k}$ acts on $\Omegapmone$ by pointwise conjugation.
\end{compactenum}
\end{definition}

\begin{remark} \label{r:aut-turn}
The fibre of $\Aut(E)$ over $b \in B$ can be identified with $SO(E_b)$, with the equivalence class 
$[\phi_b, A] \in \FSOcrossSO$ corresponding to $\phi_b \circ A \circ \phi_b ^{-1} : E_b \rightarrow E_b$. 
Similarly, the fibre of $\Turn(E)$ over $b$ consists of the turnings of $E_b$, with $[\phi_b,\gamma]$ corresponding to the path $t \mapsto \phi_b \circ \gamma(t) \circ \phi_b ^{-1}$. 
\end{remark}

A turning of a bundle $E$ restricts to a turning of each fibre and so determines a section of $\Turn(E)$. 
In this way we obtain a homeomorphism between the space of turnings of $E$ and the space of sections of 
$\Turn(E)$.  In particular we have

\begin{lemma} \label{l:TurnE}
A vector bundle $E$ is turnable if and only if $\Turn(E) \to B$ has a section.  \qed
\end{lemma}

\begin{definition} \label{d:univ-turn}
Fix a model $BSO_{2k}$ for the classifying space of oriented rank-$2k$ vector bundles
and let $VSO_{2k} \to BSO_{2k}$ denote the universal rank-$2k$ bundle.
We define $BT_{2k} := \Turn(VSO_{2k})$ to be the total space of the associated turning bundle
and let $\pi_{2k} : BT_{2k} \rightarrow BSO_{2k}$ be its projection. 
\end{definition}

\begin{remark}
The symbol $BT_{2k}$ should be read as a single unit.
Defining a topological monoid $T_{2k}$ whose classifying 
space is the associated turning bundle of the universal bundle $VSO_{2k} \to BSO_{2k}$
is an interesting question, but we will not address it in this paper.
\end{remark}

Below we explain how $BT_{2k}$ acts as a classifying space for turned vector bundles.

\begin{proposition} \label{p:tble-lift}
A rank-$2k$ bundle $E$ over a CW-complex $X$ is turnable if and only if its classifying map 
$f : X \rightarrow BSO_{2k}$ can be lifted over $\pi_{2k} \colon BT_{2k} \to BSO_{2k}$.
\end{proposition}

\begin{proof}
Since $E \cong f^*(VSO_{2k})$, we have $\Turn(E) \cong f^*(\Turn(VSO_{2k}))$; i.e.\ there is a pullback diagram
\[
\xymatrix{
\Turn(E) \ar[r] \ar[d] & 
BT_{2k} \ar[d]^-{\pi_{2k}} \\
X \ar[r]^-{f} & BSO_{2k}.
}
\]
It follows from the universal property of pullbacks that $f$ can be lifted to $BT_{2k}$ if and only if $\Turn(E)$ has a section, which is equivalent to the turnability of $E$.
\end{proof}

We now show how $BT_{2k}$ classifies rank-$2k$ turned vector bundles.

\begin{definition}
Let $VT_{2k} := \pi_{2k}^*(VSO_{2k}) \to BT_{2k}$.
\end{definition}

Note that $VT_{2k}$ has a canonical turning, which we denote $\psi^{\rm c}_t$.
Since $VT_{2k}$ is defined as a pullback of $VSO_{2k}$, its fibre over $x \in BT_{2k}$ can be identified with $(VSO_{2k})_y$, the fibre of $VSO_{2k}$ over $y = \pi_{2k}(x)$. 
By Remark \ref{r:aut-turn}, $x$ itself can be regarded as a turning of $(VSO_{2k})_y$ and hence of $(VT_{2k})_x$. That is, each fibre $(VT_{2k})_x$ of $VT_{2k}$ comes equipped with a turning (which varies continuously with $x$), showing that $\Turn(VT_{2k})$ has a canonical section.

The turned bundle $(VT_{2k} \to BT_{2k}, \psi^{\rm c}_t)$ is universal in the two senses explained in Theorem \ref{thm:VT-univ} below.

\begin{definition}
For a space $X$, let $\TB_{2k}(X)$ be the set of isomorphism classes of rank-$2k$ turned bundles over $X$: 
it consists of equivalence classes of turned bundles over $X$, where two turned bundles are equivalent if there is an isomorphism between them under which their turnings are homotopic.
\end{definition}

\begin{theorem}[$(VT_{2k}, \psi^{\rm c}_t)$ is a universal rank-$2k$ turned bundle] \label{thm:VT-univ}
\hfill
\begin{compactenum}[a)]
\item If a rank-$2k$ bundle $E$ over a CW-complex $X$ is equipped with a turning $\psi_t$, 
then there is a homotopically unique map $X \rightarrow BT_{2k}$ which induces $(E, \psi_t)$ 
from $(VT_{2k}, \psi^{\rm c}_t)$.
\item For every CW-complex $X$ there is a bijection $\TB_{2k}(X) \cong [X, BT_{2k}]$.
\end{compactenum}
\end{theorem}

\begin{proof}
a) In fact a stronger statement holds: the space of pairs $(g,\bar{g})$, where $g : X \rightarrow BT_{2k}$ is a continuous map and $\bar{g} : E \rightarrow g^*(VT_{2k})$ is an isomorphism respecting the given turnings, is contractible. To such a pair $(g,\bar{g})$ we assign a pair $(f,\bar{f})$, where $f : X \rightarrow BSO_{2k}$ is a continuous map and $\bar{f} : E \rightarrow f^*(VSO_{2k})$ is an isomorphism, by letting $f = \pi_{2k} \circ g$ and $\bar{f} = \bar{g}$ (using that $f^*(VSO_{2k}) = g^*(\pi_{2k}^*(VSO_{2k})) = g^*(VT_{2k})$). Each pair $(f,\bar{f})$ determines a pullback diagram as in the proof of Proposition \ref{p:tble-lift}. It follows from the pullback property that $f$ has a unique lift $g : X \rightarrow BT_{2k}$ corresponding to the given turning of $E$ (section of $\Turn(E)$) and if we regard $\bar{f}$ as an isomorphism $\bar{g} : E \rightarrow g^*(VT_{2k})$, then this $\bar{g}$ respects the turnings. This shows that the assignment $(g,\bar{g}) \mapsto (f,\bar{f})$ is a homeomorphism. And since $VSO_{2k}$ is a universal bundle, the space of pairs $(f,\bar{f})$ is contractible.

b) To a map $g : X \rightarrow BT_{2k}$ we assign $g^*(VT_{2k})$ with its induced turning. This way we obtain a well-defined map $[X, BT_{2k}] \rightarrow \TB_{2k}(X)$, because a homotopy of $g$ induces a bundle over $X \times I$ with a turning and after identifying this bundle with $g^*(VT_{2k}) \times I$ its turning determines a homotopy between the turnings over $X \times \{ 0 \}$ and $X \times \{ 1 \}$. It follows from Part a)that this map is surjective. 

Suppose that two maps $g_1,g_2 : X \rightarrow BT_{2k}$ determine the same element in $\TB_{2k}(X)$. This means that, after identifying $g_1^*(VT_{2k})$ with $E := g_2^*(VT_{2k})$ via some isomorphism, the induced turnings on $E$ are homotopic, i.e.\ there is a homotopy between the corresponding sections of $\Turn(E)$. This homotopy then determines a homotopy (via lifts of $\pi_{2k} \circ g_1 : X \rightarrow BSO_{2k}$) between $g_1$ and another lift $g_1' : X \rightarrow BT_{2k}$ such that under the isomorphism $g_2^*(VT_{2k}) = E \cong g_1^*(VT_{2k}) = (\pi_{2k} \circ g_1)^*(VSO_{2k}) = (\pi_{2k} \circ g_1')^*(VSO_{2k}) = (g_1')^*(VT_{2k})$ the same turning is induced on $g_2^*(VT_{2k})$ and $(g_1')^*(VT_{2k})$. By Part a), this implies that $g_1'$ is homotopic to $g_2$. Therefore the map $[X, BT_{2k}] \rightarrow \TB_{2k}(X)$ is also injective, hence it is a bijection.
\end{proof}

\begin{remark} \label{r:assoc-h}
In the constructions of this section, instead of $\Omegapmone$ we could use one of its connected components, 
$\Omega_\gamma SO_{2k}$ for a $\gamma \in \Omegapmone$.
Then the turning bundle $\Turn(E)$ would be replaced with its 
subbundle $\Turn^{\gamma}(E)$ and $BT_{2k}$ with its connected component $BT^{\gamma}_{2k} = \Turn^{\gamma}(VSO_{2k})$. 
A bundle $E$ over a $CW$-complex $X$ is 
$\gamma$-turnable if and only if $\Turn^{\gamma}(E)$ has a section and if and only if its classifying map $f \colon X \to BSO_{2k}$ can be lifted to $BT^{\gamma}_{2k}$. Moreover, 
$VT^{\gamma}_{2k} := VT_{2k} \big| _{BT^{\gamma}_{2k}}$ is universal among bundles equipped with a  
$\gamma$-turning.
\end{remark}

\subsection{The gauge group}  \label{ss:GG}
For an oriented rank-$2k$ vector bundle $\pi \colon E \to B$, recall that $\Fr(E)$ denotes the 
frame bundle of $E$, which is a principal $SO_{2k}$-bundle over $B$.
As an elementary exercise in linear algebra shows, the space of automorphisms of a vector bundle $E \to B$
is canonically homeomorphic to the gauge group of $\Fr(E)$, as defined in \cite[Ch.\ 7]{h94} and we will use
these topological groups interchangeably, denoting them by $\GG_E$.
In this section we relate the existence of turnings on $E$ to the topology of $\GG_E$.

The automorphisms $\idbb_E$ and $-\idbb_E$ define elements of $\GG_E$.
Specifically, since $-\idbb$ lies in $Z(SO_{2k})$, the centre of $SO_{2k}$, we obtain the global map
$-\idbb_E \in \GG_E,\,p \mapsto p(-\idbb)$.
Considering $[\idbb_E], [-\idbb_E] \in \pi_0(\GG_E)$, we see from
Definition~\ref{d:Tbundle} that $E$ is turnable if and only 
if $[-\idbb_E] = [\idbb_E] \in \pi_0(\GG_E)$.
Indeed, somewhat more is true as we now explain.

Fixing a frame $p \in \Fr(E)$ over $b = \pi(p)$ and restricting to the fibre of $\Fr(E) \to B$ over $b$, 
we obtain a continuous homomorphism of topological groups $r_p \colon \GG_E \to SO_{2k}$.
Replacing $SO_{2k}$ by the mapping cylinder of $r_p$, we regard $r_p$ as an inclusion and 
consider the pair $(SO_{2k}, \GG_E)$.
A path $\gamma \in \Omega_{\pm \idbb} SO_{2k}$ defines an element
$[\gamma]_\GG \in \pi_1(SO_{2k}, \GG_E)$, by identifying $\idbb_E, -\idbb_E \in \GG_E$ with $r_p(\idbb_E)=\idbb, r_p(-\idbb_E)=-\idbb \in SO_{2k}$ respectively and viewing $\gamma$ as a path in $SO_{2k}$ connecting $\idbb_E$ and $-\idbb_E$.
Since $r_p \colon \GG_E \to SO_{2k}$ is a homomorphism of topological groups,
$\pi_1(SO_{2k}, \GG_E)$ inherits a group structure from the group structures on $\GG_E$ and $SO_{2k}$
and we denote the unit by $e$.
Then we have

\begin{lemma} \label{l:h_and_barh}
A bundle $E$ is $\gamma$-turnable if and only if $[\gamma]_\GG = e \in \pi_1(SO_{2k}, \GG_E)$. \qed
\end{lemma}

Given the above, it is natural to consider the final segment of the homotopy long exact sequence
of the pair $(SO_{2k}, \GG_E)$, which runs as follows:
\begin{equation} \label{eq:LES}
 \dots \to \pi_1(\GG_E) \xra{~(r_p)_*~} \pi_1(SO_{2k}) \xra{~~~} \pi_1(SO_{2k}, \GG_E) \xra{~~~}
\pi_0(\GG_E) \to 0
\end{equation}
Now $[\bar \beta]_\GG = [\beta]_\GG + [\eta]$, where $+$ denotes the natural action of $\pi_1(SO_{2k})$
on $\pi_1(SO_{2k}, \GG_E)$ and $[\eta] \in \pi_1(SO_{2k})$ is the generator.
We see that $[\beta]_\GG = [\bar \beta]_\GG$ if and only 
if $(r_p)_* \colon \pi_1(\GG_E) \to \pi_1(SO_{2k})$ is onto.  For example, in Section \ref{ss:S4} we shall see that there
are rank-$4$ bundles $E \to S^4$, which are $\beta$-turnable but not $\bar \beta$-turnable.
Applying Lemma~\ref{l:h_and_barh} we see that for these bundles $[\beta]_\GG \neq [\bar \beta]_\GG
\in \pi_1(SO_{2k}, \GG_E)$ and hence the map $(r_p)_* \colon \pi_1(\GG_E) \to \pi_1(SO_{2k})$ is zero.
In fact, the homomorphism $\pi_1(\GG_E) \to \pi_1(SO_{2k})$ is closely related to
the ``turning obstruction for the essential loop in $SO_{2k}$" and 
we next discuss turnings in a more general setting.

\subsection{The central groupoid and general turnings}  \label{ss:gt}
In this subsection we generalise the definition of a turning.
Let $G$ be a path-connected topological group with centre $Z(G)$.  
For the computations in this paper the groups $SO_{2k}$, their double covers $Spin_{2k}$ and their quotients $PSO_{2k} := SO_{2k}/\{\pm \idbb\}$ will be relevant and we will consider these groups in more detail at the end of this subsection.

Let $\pi \colon P \to B$ be a principal $G$-bundle over a path-connected space $B$.
The gauge group of $P$, denoted $\GG_P$, is the group of $G$-equivariant
fibrewise automorphisms of $P$.  
Given $z \in Z(G)$, fibrewise multiplication by $z$ defines a central element $z_P \in Z(\GG_P)$,
where for all $p \in P$
\[ z_P(p) := p \cdot z.\]
We note that if $Z(G)$ is discrete, then the map 
$Z(G) \to Z(\GG_P), z \mapsto z_P$, is an isomorphism.
We shall be interested in paths $\gamma \colon I \to G$ which start and end at elements of the centre $Z(G)$
and whether they can be lifted to paths in $\GG_P$ which start and end at $\gamma(0)_P$ and $\gamma(1)_P$.
Hence we make the following
\begin{definition}[Central groupoid]
The {\em central groupoid} of $G$ is the restriction of the fundamental 
groupoid of $G$ to paths which start and end in the centre of $G$.
We will use $\pi^Z(G)$ to ambiguously denote the central groupoid of $G$ or the set of its morphisms.
\end{definition}

\begin{remark} \label{r:pi^Z(G)}
We note that point-wise multiplication gives $\pi^Z(G)$, the set of morphisms of the central groupoid,
a group structure and there is short exact sequence
\[ 1 \to \pi_1(G, e) \to \pi^Z(G) \to Z(G) \times Z(G) \to 1,\]
where $\pi^Z(G) \to Z(G) \times Z(G)$ is defined by $[\gamma] \mapsto (\gamma(0), \gamma(1))$
and $e \in G$ is the identity.  While we do not use this group structure in what follows,
it may be helpful for understanding $\pi^Z(G)$; e.g.\ it shows that $\pi^Z(SO_{2k})$ has $8$ morphisms.
\end{remark}

For a point $p \in P$, let $b = \pi(p)$ and $P_b := p \cdot G$ be the fibre of $P \to B$ over $b$.
We define the {\em restriction map}
\[ r_p \colon \GG_P \to G, \]
by restricting elements of the gauge group to the fibre over $b$
and using the equation
\[ \phi(p) = p \cdot r_p(\phi)\]
for all $\phi \in \GG_P$.
If we vary $p \in P_b$, then $r_{p \cdot g}(\phi) = g^{-1}r_p(\phi) g$ for all $\phi \in \GG_P$ and $g \in G$.
Recalling that $G$ and $B$ are path-connected, we see for any path $\phi_t \colon I \to \GG_P$
with $\phi_0 = z_P$ and $\phi_1 = z'_P$, that $[r_p(\phi_t)] \in \pi^Z(G)$ is independent 
of the choice of $p$.

\begin{definition}
Let $\phi_t \colon I \to \GG_P$ be a path such that $\phi_0 = z_P$ and $\phi_1 = z'_P$ for some $z,z' \in Z(G)$. We define $r(\phi_t) \in \pi^Z(G)$ to be $[r_p(\phi_t)]$ for any $p \in P$.
\end{definition}

\begin{definition}[$\gamma$-turning and $\gamma$-turnable]
Let $[\gamma] \in \pi^Z(G)$ be represented by a path $\gamma \colon I \to G$.
A {\em $\gamma$-turning} of a principal $G$-bundle $P$ is a path 
$\phi_t \colon I \to \GG_P$ with $\phi_0 = \gamma(0)_P, \phi(1) = \gamma(1)_P$ and
$r(\phi_t) = [\gamma] \in \pi^Z(G)$.
If $P$ admits a $\gamma$-turning then $P$ is called {\em $\gamma$-turnable}.
\end{definition}

We end this subsection by considering the groups $SO_{2k}$, $Spin_{2k}$ and $PSO_{2k}$. When $2k > 2$, we have $PSO_{2k} = SO_{2k}/Z(SO_{2k}) \cong Spin_{2k}/Z(Spin_{2k})$ and we list the centres and fundamental groups of these groups in the following tables (where $j \geq 1$), which follow from Lemma~\ref{l:pi_1PSO_2k} below:
\[
\begin{array}{c|c|c}
G & Z(G) & \pi_1(G) \\
\hline
Spin_{4j} & \Z/2 \oplus \Z/2 & \{e\} \\
SO_{4j} & \Z/2 & \Z/2 \\
PSO_{4j} & \{e\} & \Z/2 \oplus \Z/2 
\end{array}
\qquad \qquad \qquad 
\begin{array}{c|c|c}
G & Z(G) & \pi_1(G) \\
\hline
Spin_{4j+2} & \Z/4 & \{e\} \\
SO_{4j+2} & \Z/2 & \Z/2 \\
PSO_{4j+2} & \{e\} & \Z/4 
\end{array}
\]

\noindent
The next lemma is well-known but we include its proof to further illustrate the structure of the central groupoid of $SO_{2k}$.

\begin{lemma} \label{l:pi_1PSO_2k}
If $k \geq 2$, then $Z(PSO_{2k}) = \{ e \}$, $Z(Spin_{2k}) \cong \pi_1(PSO_{2k})$ and
\[
\pi_1(PSO_{2k}) \cong 
\begin{cases}
\Z/2 \oplus \Z/2 & \text{if $k$ is even,} \\
\Z/4 & \text{if $k$ is odd.}
\end{cases}
\]
\end{lemma}

\begin{proof}
To see that the centre of $PSO_{2k}$ is trivial, let $x \in SO_{2k}$ lie in the preimage of $Z(PSO_{2k})$.
Then the commutator $[x, \cdot\,]$ defines a map $SO_{2k} \rightarrow Z(SO_{2k})$. Since $Z(SO_{2k})$ is discrete and $[x,\idbb]=\idbb$, this is the constant $\idbb$ map.
Hence $x \in Z(SO_{2k})$ and thus $Z(PSO_{2k}) = \left\{ e \right\}$. 

If $q \colon Spin_{2k} \to SO_{2k}$ denotes the non-trivial double covering then 
we see that $Z(Spin_{2k}) = q^{-1}(Z(SO_{2k}))$.
Now $Z(SO_{2k}) = \{\pm \idbb\}$, the covering $q' \colon Spin_{2k} \to PSO_{2k}$ 
is the universal covering of $PSO_{2k}$ and it is the composition of $q$ and $SO_{2k} \to PSO_{2k}$.
It follows that $Z(Spin_{2k}) = (q')^{-1}([\idbb]) \cong \pi_1(PSO_{2k})$.

To compute $\pi_1(PSO_{2k})$ we first consider the central groupoid $\pi^Z(SO_{2k})$. It is generated by the morphisms $[\beta], [\bar{\beta}] : \idbb \rightarrow -\idbb$, subject to the relation $([\bar{\beta}]^{-1} \circ [\beta])^2 = \id_{\idbb}$; i.e.\ $[\bar{\beta}]^{-1} \circ [\beta] = [\eta]$ is the generator of $\pi_1(SO_{2k}) = \Z/2$. 
The following diagram shows the named morphisms in $\pi^Z(SO_{2k})$:
\[
\xymatrix{
-\idbb & \idbb \ar@/_/[l]_{[\beta]} \ar@/^/[l]^{[\bar{\beta}]} \ar@(dr,ur)[]_{[\eta]}
}
\]
The projection $SO_{2k} \to PSO_{2k}$ induces a surjective map of groupoids $\pi^Z(SO_{2k}) \rightarrow \pi_1(PSO_{2k})$, which sends two morphisms $[\gamma], [\gamma'] \in \pi^Z(SO_{2k})$ into the same element of 
$\pi_1(PSO_{2k})$ if and only if $[\gamma'] = [-\gamma]$, where $-\gamma$ denotes the path $\gamma$ multiplied pointwise by $-\idbb$.  Since 
\[
[-\beta] = 
\begin{cases}
[\beta]^{-1} & \text{if $k$ is even,} \\
[\bar{\beta}]^{-1} & \text{if $k$ is odd,}
\end{cases}
\]
the computation of $\pi_1(PSO_{2k})$ follows.
\end{proof}

\subsection{$\eta$-turnings and the path components of $\GG_E$} \label{ss:gg-eta}
Recall that $\eta$ denotes the generator of $\pi_1(SO_{2k})$ and that
by definition a rank-$2k$ vector bundle $E \to B$ is $\eta$-turnable
if and only if the restriction induces a surjective map $\pi_1(\GG_E) \to \pi_1(SO_{2k})$.
We return to our discussion of the exact sequence \eqref{eq:LES} from Section~\ref{ss:GG}
and first identify it with an isomorphic exact sequence.
Assuming that $b \in B$ is non-degenerate, 
the homomorphism $r_p \colon \GG_E \to SO_{2k}$
is onto and there is a short exact sequence of topological groups
\begin{equation} \label{eq:LES1}
  \GG_{E, 0} \xra{~~~} \GG_E \xra{r_p} SO_{2k}, 
\end{equation}
where, by definition $\GG_{E, 0} := \mathrm{Ker}(r_p) \subset \GG_E$.
Regarding~\eqref{eq:LES1} as a principal $\GG_{E, 0}$-bundle, it is classified by a map $SO_{2k} \to B\GG_{E, 0}$
such that 
\begin{equation} \label{eq:BLES1}
 \GG_E \xra{r_p} SO_{2k} \to B\GG_{E, 0} 
\end{equation} 
is a fibration sequence.
The homotopy long exact sequence of \eqref{eq:LES}
maps isomorphically to the homotopy long exact sequences of~\eqref{eq:BLES1} and~\eqref{eq:LES1} as follows:
\begin{equation} \label{eq:LES2}
\xymatrix{
 \pi_1(\GG_E)  \ar[d]^= \ar[r]^-{(r_p)_*} &
 \pi_1(SO_{2k})  \ar[d]^= \ar[r] &
  \pi_1(SO_{2k}, \GG_E) \ar[d]^\cong \ar[r] &
\pi_0(\GG_E)  \ar[d]^= \ar[r] & 
0   \\
\pi_1(\GG_E) \ar[d]^= \ar[r]^-{(r_p)_*} &
\pi_1(SO_{2k}) \ar[d]^= \ar[r] &
\pi_1(B\GG_{E, 0}) \ar[d]^\cong \ar[r] &
\pi_0(\GG_E) \ar[d]^= \ar[r] &
0 
\\
 \pi_1(\GG_E)  \ar[r]^-{(r_p)_*} &
 \pi_1(SO_{2k}) \ar[r] &
  \pi_0(\GG_{E, 0}) \ar[r] &
\pi_0(\GG_E) \ar[r] & 
0} 
\end{equation}

Now we fix the base space $B$ and suppose that $B = S X$ is a suspension.
For any rank-$2k$ vector bundle $E \to SX$,
there is a homotopy equivalence $\GG_{E, 0} \simeq \mathrm{Map}((SX, \ast), (SO_{2k}, \idbb))$
and in particular the homotopy type of $\GG_{E, 0}$ does not depend on the vector bundle $E$.
If $|\pi_0(\GG_{E, 0})| = \big| [SX, SO_{2k}] \big|$ is finite then the exact 
sequences above in \eqref{eq:LES2} show that $|\pi_0(\GG_E)|$ depends on the $\eta$-turnability
of $E$.  Specifically we have the following

\begin{theorem} \label{t:gauge-eta}
If $B = SX$ is a suspension and $n_{SX} := \left| [SX, SO_{2k}] \right|$ is finite,
then for any rank-$2k$ vector bundle $E \to SX$
\[ 
\bigl| \pi_0(\GG_E) \bigr| = 
\begin{cases}
n_{SX} & \text{if $E$ is $\eta$-turnable,} \\
\frac{n_{SX}}{2} & \text{if $E$ is not $\eta$-turnable.}
\end{cases}\]
In particular, the $\eta$-turnability of rank-$2k$ vector bundles over $SX$ with $n_{SX}$ finite is a homotopy invariant of
the gauge groups of these bundles. \qed
\end{theorem}

\begin{remark}
As one might expect,
if a rank-$2k$ vector bundle $E$ admits a spin structure,
then the $\eta$-turnability of $E$ is equivalent to the $\gamma_{z_0}$-turnability of the associated principal $Spin_{2k}$-bundle for a certain path $\gamma_{z_0}$ in $Spin_{2k}$.
Let $e \in Spin_{2k}$ be the identity and define $z_0 \in Z(Spin_{2k}) \setminus \{ e \}$ to 
be the unique element mapping to $\idbb \in SO_{2k}$.
Since $Spin_{2k}$ is simply connected, 
there is a unique path homotopy class of paths from $e$ to $z_0$, we denote this path by $\gamma_{z_0}$.
Given a principal $Spin_{2k}$-bundle $P \to B$, we can consider the $\gamma_{z_0}$-turning 
problem for $P$.
As $Spin_{2k}$ acts on $\R^{2k}$ via the double covering $Spin_{2k} \to SO_{2k}$
and the standard action of $SO_{2k}$, there is a rank-$2k$ vector bundle $E_P : = P \times_{Spin_{2k}} \R^{2k}$ associated to $P$
and it is not hard to see that the following statements are equivalent:
\begin{compactenum}
\item The principal $Spin_{2k}$-bundle $P \to B$ is $\gamma_{z_0}$-turnable;
\item The vector bundle $E_P$ associated to $P$ is $\eta$-turnable;
\item The map $(r_p)_* \colon \pi_1(\GG_{E_P}) \to \pi_1(SO_{2k})$ is onto.
\end{compactenum}
\end{remark}

\section{The turning obstruction}  \label{s:to}

In this section we define the turning obstruction for bundles over suspensions. 

First we consider rank-$2k$ oriented vector bundles. Over a suspension $SX$ such a bundle corresponds to an element of $[X, SO_{2k}]$. Given a path $\gamma \in \Omegapmone$ we first define a map $\To_{\gamma} : [X, SO_{2k}] \rightarrow [X, \Omega_\gamma SO_{2k}]$ (where $\Omega_\gamma SO_{2k} \subset \Omegapmone$ denotes the connected component of $\gamma$, see Definition \ref{d:omg_gmm}) and prove that it is a complete obstruction to the $\gamma$-turnability of a bundle. We also introduce variants $\overline{\To}_\gamma \colon [X, SO_{2k}] \to [X, \Omega_0 SO_{2k}]$ and $\TO_\gamma, \ol{\TO}_\gamma \colon [X, SO_{2k}] \to [SX, SO_{2k}]$ and show that they are equivalent to $\To_{\gamma}$. Finally we prove that these maps are homomorphisms if $X$ is a suspension.

In Section~\ref{ss:gen} we consider bundles with a path-connected structure group $G$. Given a path $\gamma$ between elements of the centre $Z(G)$, we define a generalised turning obstruction map $\To_{\gamma} : [X, G] \rightarrow [X, \Omega_\gamma G]$. If we fix an element $[g] \in [X, G]$, then we can regard $\To_{\gamma}([g])$ as a function of $\gamma$ and we show that it is compatible with concatenation of paths. We also consider a normalised version of the turning obstruction, $\overline{\To}_\gamma \colon [X, G] \to [X, \Omega_0 G]$. This allows us to compare turning obstructions for different paths and we find that $\overline{\To}_{\gamma} = \overline{\To}_{a\gamma}$ for any $a \in Z(G)$ (where $(a\gamma)(t)=a\gamma(t)$). When $Z(G)$ is discrete, we introduce the quotient $PG = G/Z(G)$ and use these observations to show that the turning obstructions are determined by a map $\widehat{\To}_{\cdot}(\cdot) \colon \pi_1(PG) \times [X, G] \rightarrow [X, \Omega_0 G]$, which is a homomorphism in the first variable (and also in the second one, if $X$ is a suspension). As an application we prove Theorem \ref{t:to-ord24}.

\subsection{The turning obstruction for vector bundles} \label{ss:tovb}

Let $X$ be a CW-complex and let $C_0X$ and $C_1X$ be two copies of the cone on $X$, so that 
$SX = C_0X \cup_X C_1X$. 
By~\cite[Ch.\ 8, Theorem 8.2]{h94}, the set of isomorphism classes of oriented rank-$2k$ bundles over $SX$ 
is in bijection with $[X, SO_{2k}]$, the set of homotopy classes of maps from $X$ to $SO_{2k}$.
A bundle $E$ corresponds to its {\em clutching function} $g \colon X \rightarrow SO_{2k}$ 
between two local trivialisations $\varphi_i : C_iX \times \R^{2k} \rightarrow E \big| _{C_iX}$, 
defined by $\varphi_0^{-1}\big|_X \circ \varphi_1 \big|_X(x, v) = (x, g(x)v)$.

\begin{definition}[The $\gamma$-turning obstruction] \label{d:to}
Let $\gamma \in \Omegapmone$.
\begin{compactenum}[a)]
\item We define the map $\rho_{\gamma} : SO_{2k} \rightarrow \Omega_\gamma SO_{2k}$ by $\rho_{\gamma}(A) = (t \mapsto A\gamma(t)A^{-1})$.
\item For any CW-complex $X$ the {\em $\gamma$-turning obstruction map} is 
$\To_{\gamma} := (\rho_{\gamma})_* : [X, SO_{2k}] \rightarrow [X, \Omega_\gamma SO_{2k}]$.
\end{compactenum}
\end{definition}

\noindent
Let $0 \in [X, \Omega_\gamma SO_{2k}]$ denote the homotopy class of the constant map.
Definition~\ref{d:to} is justified by the following

\begin{proposition} \label{p:to}
Let $E$ be an oriented rank-$2k$ bundle over $SX$ with clutching function $g \colon X \to SO_{2k}$.
 
Then $E$ is $\gamma$-turnable if and only if $\To_{\gamma}([g]) = 0 \in [X, \Omega_\gamma SO_{2k}]$. 
\end{proposition}

\begin{proof}
Recall that $\gamma$-turnings of $E$ can be identified with sections of the associated $\gamma$-turning bundle $\Turn^{\gamma}(E) = \Fr(E) \times _{SO_{2k}} \Omega_{\gamma} SO_{2k}$ (see Lemma \ref{l:TurnE} and Remark \ref{r:assoc-h}). The local trivialisations $\varphi_i$ of $E$ induce local trivialisations $\bar{\varphi}_i : C_iX \times \Omega_{\gamma} SO_{2k} \rightarrow \Turn^{\gamma}(E) \big| _{C_iX}$ of $\Turn^{\gamma}(E)$. 
By construction, the clutching function $g \colon X \to SO_{2k}$ is also the clutching function of 
$\Turn^{\gamma}(E)$ (recall that $SO_{2k}$ acts on 
$\Omega_{\gamma} SO_{2k}$ by pointwise conjugation).

Since $C_iX$ is contractible and $\Omega_{\gamma} SO_{2k}$ is connected, each restriction $\Turn^{\gamma}(E) \big| _{C_iX}$ has a unique section $s_i : C_iX \rightarrow \Turn^{\gamma}(E) \big| _{C_iX}$ up to homotopy, 
given by $s_i(y) = \bar{\varphi}_i(y,\gamma)$. 
Hence a global section of $\Turn^{\gamma}(E)$ exists if and only if $s_0 \big| _X$ and $s_1 \big| _X$ are homotopic sections. For $x \in X$ we have 
\[ s_0(x) = \bar{\varphi}_0(x,\gamma)
\quad \text{and} \quad
s_1(x) = \bar{\varphi}_1(x,\gamma) = \bar{\varphi}_0 \circ \bar{\varphi}_0^{-1} \circ \bar{\varphi}_1(x,\gamma) = \bar{\varphi}_0(x,\gamma^{g(x)}) = \bar{\varphi}_0(x,\rho_\gamma(g(x))),\]
where $\gamma^{g(x)}$ denotes the action of $g(x) \in SO_{2k}$ on $\gamma \in \Omega_{\gamma} SO_{2k}$. These sections are homotopic if and only if 
$\rho_{\gamma} \circ g : X \rightarrow \Omega_{\gamma} SO_{2k}$ 
is homotopic to the constant map with value $\gamma$, i.e.\ if and only if $\To_{\gamma}([g]) = 0$.
\end{proof}

The turning obstruction is a map of pointed sets, but $[X, SO_{2k}]$ is a group and we will show that $[X, \Omega_\gamma SO_{2k}]$ can also be equipped with a group structure and that $\To_\gamma$ and related maps are often group homomorphisms; see Lemma \ref{l:to+bto}.

Let $\Omega_0 SO_{2k} \subset \Omega SO_{2k}$ denote the component of contractible loops and define the homeomorphism
\[ p_\gamma \colon \Omega_\gamma SO_{2k} \to \Omega_0 SO_{2k},
\quad \delta \mapsto (t \mapsto \delta(t)\gamma(t)^{-1}),\]
as well as the commutator map $\bar \rho_\gamma := p_\gamma \circ \rho_\gamma$,
\[ \bar \rho_\gamma \colon SO_{2k} \to \Omega_0SO_{2k},
\quad A \mapsto (t \mapsto A\gamma(t)A^{-1}\gamma(t)^{-1}).\]

\begin{definition}[Normalised $\gamma$-turning obstruction]
Let $\gamma \in \Omegapmone$. For any $CW$-complex $X$, the {\em normalised $\gamma$-turning
obstruction map} is $\overline{\To}_\gamma: = (\bar \rho_\gamma)_* \colon [X, SO_{2k}] \to [X, \Omega_0 SO_{2k}]$.
\end{definition}

Since $p_\gamma$ is a homeomorphism, the induced map $(p_{\gamma})_* \colon [X, \Omega_{\gamma} SO_{2k}] \to [X, \Omega_0 SO_{2k}], [h] \mapsto [p_\gamma \circ h]$ is a bijection which preserves $0$ and hence 
an oriented rank-$2k$ bundle $E \to SX$ with clutching function $g \colon X \to SO_{2k}$ is 
$\gamma$-turnable if and only if $\ol{\To}_\gamma([g]) = 0$.
For computing $\To_\gamma$ and $\ol{\To}_\gamma$ it is useful to consider their adjointed versions, which we will define below. 

\begin{definition}[Forgetful adjoints]
Let $h \colon X \to \Omega_\gamma SO_{2k}$ and $h' \colon X \to \Omega_0 SO_{2k}$ be continuous maps. By taking their adjoints
\[ \ad(h) \colon SX \to SO_{2k},
\quad [x, t] \mapsto h(x)(t) 
\qquad \text{and} \qquad
\ad(h') \colon SX \to SO_{2k},
\quad [x, t] \mapsto h'(x)(t)\]
we define the {\em forgetful adjoint maps} $\ad : [X, \Omega_{\gamma} SO_{2k}] \rightarrow [SX, SO_{2k}]$ and $\ad : [X, \Omega_0 SO_{2k}] \rightarrow [SX, SO_{2k}]$. (We call these maps ``forgetful'' because $[\ad(h)]$ and $[\ad(h')]$ are regarded as elements of $[SX, SO_{2k}]$ rather than of the more restricted sets of homotopy classes on which the inverse adjoint maps $[\ad(h)] \mapsto [h]$ and $[\ad(h')] \mapsto [h']$ are naturally defined.)
\end{definition}

If $X$ is connected, then the forgetful adjoint maps are bijections and they preserve $0$, the homotopy class of the constant map.

\begin{definition}[Adjointed $\gamma$-turning obstructions] \label{def:adj-to}
Define $\TO_\gamma : = \ad \circ \To_\gamma \colon [X, SO_{2k}] \to [SX, SO_{2k}]$
and $\ol{\TO}_\gamma : = \ad \circ \ol{\To}_\gamma \colon [X, SO_{2k}] \to [SX, SO_{2k}]$.
\end{definition}

\begin{lemma} \label{l:to+bto}
\begin{compactenum}[a)]
\item Let $X$ be a $CW$-complex. Then $\TO_\gamma = \ol{\TO}_\gamma : [X, SO_{2k}] \to [SX, SO_{2k}]$.
\item If $X$ is a suspension, then $\To_\gamma, \ol{\To}_\gamma, \TO_\gamma$ and $\ol{\TO}_\gamma$ are each homomorphisms of abelian groups.
\end{compactenum}
\end{lemma}

\begin{proof}
a) We show that the diagram
\[ 
\xymatrix@R=15pt@C=20pt{
 & [X, \Omega_\gamma SO_{2k}] \ar[dd]^{(p_{\gamma})_*} \ar[r]^{\ad} & [SX, SO_{2k}] \ar@{=}[dd] \\
[X, SO_{2k}] \ar[ur]^{\To_\gamma} \ar[dr]_{\ol{\To}_\gamma} & & \\
 & [X, \Omega_0 SO_{2k}] \ar[r]^{\ad} & [SX, SO_{2k}]
}
\]
commutes.
The left hand triangle commutes be definition.
For the right hand square,
consider the path of paths $s \mapsto \gamma_s$, where $\gamma_s : I \rightarrow SO_{2k}$ is defined by $\gamma_s(t) = \gamma(st)$ for $s, t \in I$.
Then the map
\[ H \colon SX \times I \to SO_{2k}, 
\quad
([x, t], s) \mapsto g(x) \gamma(t) g(x)^{-1} \gamma_s(t)^{-1}
\]
is a homotopy from $\ad(\rho_\gamma \circ g)$ to $\ad(\bar \rho_\gamma \circ g)$,
which proves that the square commutes.

b) First note that $\Omega_0 SO_{2k}$ is a topological group (via pointwise multiplication of loops), so we can use the homeomorphism $p_\gamma$ to get a topological group structure on $\Omega_\gamma SO_{2k}$. Hence for $H = SO_{2k}$, $\Omega_0 SO_{2k}$ or $\Omega_\gamma SO_{2k}$ and any space $Y$ the set $[Y, H]$ inherits a group structure from $H$ (and with these $(p_{\gamma})_*$ is automatically an isomorphism). 

If $Y$ is pointed, then the set $[Y, H]_*$ of homotopy classes of basepoint-preserving maps is also a group and the forgetful map $[Y, H]_* \to [Y, H]$ is an isomorphism. If $Y$ is a suspension, then for any space $Z$ the set $[Y,Z]_*$ has a group structure coming from the co-H-space structure on $Y$ and on the sets $[Y, H]_*$ the two group structures coincide and they are abelian. Since the group structure on $[Y,Z]_*$ is natural in $Z$, it follows that if $X$ is a suspension, then the maps $\To_\gamma$ and $\ol{\To}_\gamma$ (which are induced by maps of spaces) are homomorphisms (and all groups involved are abelian).

Finally, loop concatenation gives $\Omega_0 SO_{2k}$ an H-space structure and the induced group structure on $[X, \Omega_0 SO_{2k}]$ coincides with the previously defined one. By comparing this with the group structure on $[SX, SO_{2k}]$ coming from the suspension $SX$, we obtain that the forgetful adjoint maps and hence $\TO_\gamma$ and $\ol{\TO}_\gamma$, are also homomorphisms.
\end{proof}

\begin{remark}
By Lemma~\ref{l:to+bto} b), if $X = SY$ is a suspension, then the set of isomorphism classes of $\gamma$-turnable bundles over $SX = S^2Y$ can be identified with a subgroup of $[X, SO_{2k}]$.
\end{remark}

\begin{question}
\begin{compactenum}[a)]
\item The sets $[X, SO_{2k}]$ and $[SX, SO_{2k}]$ have natural group structures even when $X$ is not a suspension (or co-H-space). Is $\TO_\gamma$ a group homomorphism for an arbitrary $X$? 
\item Isomorphism classes of rank-$2k$ oriented bundles over a space $B$ are in bijection with $[B,BSO_{2k}]$, so when $B$ is a suspension, $\TO_\gamma$ can be regarded as a map $[B,BSO_{2k}] \rightarrow [B,SO_{2k}]$. Is there a similar $\gamma$-turning obstruction map for bundles over an arbitrary space $B$?
\end{compactenum}
\end{question}

We next briefly discuss the behaviour of the turning obstruction 
under stabilisation: we will return to this topic in greater detail in Section~\ref{s:stable}.
Let $i \colon SO_{2k} \to SO_{2k+2}$ denote the standard inclusion
and let $S = i_* \colon [X, SO_{2k}] \to [X, SO_{2k+2}]$ denote the stabilisation
map induced by $i$.
Given a path $\gamma_0 \in \Omega_{\pm \idbb} SO_{2}$, taking the orthogonal sum with $\gamma_0$ defines a map $i_{\gamma_0} \colon \Omega_\gamma SO_{2k} \to \Omega_{\gamma \oplus \gamma_0} SO_{2k+2}$.
It is clear from the definitions that the turning obstructions satisfy 
$\To_{\gamma \oplus \gamma_0}([i \circ g]) = i_{\gamma_0 *}(\To_\gamma([g]))$
and indeed we have

\begin{lemma} \label{l:to+stab}
Let $X$ be a $CW$-complex and $g \colon X \to SO_{2k}$ a map.
Then the adjointed $\gamma$-turning obstruction satisfies 
$\TO_{\gamma \oplus \gamma_0}([i \circ g]) = S\big( \TO_\gamma([g]) \bigr)$. 
In particular, $\TO_\beta([i \circ g]) = \TO_{\bar \beta}([i \circ g])$.
\qed
\end{lemma} 

\begin{proof}
Write 
$A \oplus B \in SO_{2k+2}$ for the block sum of matrices $A \in SO_{2k}$ and $B \in SO_2$
and consider the path of paths $s \mapsto (\gamma_0)_s$, where $(\gamma_0)_s : I \rightarrow SO_2$ is defined by $(\gamma_0)_s(t) = \gamma_0(st)$ for $s, t \in I$.
Then
\[ H \colon SX \times I \to SO_{2k+2}, 
\quad
([x, t], s) \mapsto g(x) \gamma(t) g(x)^{-1} \oplus (\gamma_0)_{(1-s)}(t)
\]
is a homotopy from $\ad(\rho_{\gamma \oplus \gamma_0} \circ (i \circ g))$ to $i \circ \ad(\rho_\gamma(g))$,
which proves the first statement of the lemma. In particular, $\TO_{\gamma \oplus \gamma_0}([i \circ g])$ is independent of the choice of $\gamma_0 \in \Omega_{\pm \idbb} SO_2$. Since the map $\pi_0(\Omega_{\pm \idbb} SO_2) \rightarrow \pi_0(\Omega_{\pm \idbb} SO_{2k+2})$, $[\gamma_0] \mapsto [\gamma \oplus \gamma_0]$ is surjective (for any $\gamma \in \Omegapmone$), the second statement follows.
\end{proof}

We conclude this subsection with a remark on a related point of view on the turning obstruction.

\begin{remark}
Consider the associated $\gamma$-turning bundle of the universal bundle $VSO_{2k}$ 
(see Definition \ref{d:univ-turn} and Remark \ref{r:assoc-h}) and the Puppe sequence
\[
\ldots \to \Omega BT^{\gamma}_{2k} \to \Omega 
BSO_{2k} \to \Omega_{\gamma} SO_{2k} \to BT^{\gamma}_{2k} \to BSO_{2k},
\]
where $\Omega$ denotes the based loops functor.
After applying the functor $[X,-]_*$ to this sequence 
and the adjunction $[X, \Omega Y]_* \cong [\Sigma X,Y]_*$, where $\Sigma X$ denotes the reduced suspension,
we obtain an exact sequence
\[
\ldots \to [\Sigma X, BT^{\gamma}_{2k}]_* \to [\Sigma X, BSO_{2k}]_* \xrightarrow{~\partial~} 
[X, \Omega_{\gamma} SO_{2k}]_* \to [X, BT^{\gamma}_{2k}]_* \to [X, BSO_{2k}]_*.
\]
The arguments in the proof of Proposition \ref{p:to} also show that the $\gamma$-turning obstruction can be identified with the boundary map $\partial$ (note that $[\Sigma X, BSO_{2k}]_* \cong [X, SO_{2k}]_* \cong [X, SO_{2k}]$ and $[X, \Omega_\gamma SO_{2k}]_* \cong [X, \Omega_\gamma SO_{2k}]$).
So by using the exactness of the sequence we obtain an alternative proof of Proposition \ref{p:tble-lift} for bundles over suspensions.
\end{remark}

\subsection{General turning obstructions} \label{ss:gen}
In this subsection we define the turning obstruction for the general turnings of principal bundles,
as in Section~\ref{ss:gt}.
This will help us establish some basic properties of the turning obstruction for vector bundles. 

\begin{definition}
Let $G$ be a path-connected topological group, $a,b \in G$ arbitrary elements and $\gamma : I \rightarrow G$ a path in $G$. We introduce the following notation: 
\begin{compactitem}
\item $\Omega G$ denotes the space of loops in $G$ based at the identity element;
\item $\Omega_0 G \subseteq \Omega G$ is the space of nullhomotopic loops (i.e.\ the connected component of the constant loop);
\item $\Omega_{a,b} G$ is the space of paths in $G$ from $a$ to $b$;
\item $\Omega_{\gamma} G$ is the space of paths homotopic (rel $\partial I$) to $\gamma$ (i.e.\ the connected component of $\gamma$ in $\Omega_{\gamma(0),\gamma(1)} G$);
\item $\Omega_Z G = \bigcup_{a,b \in Z(G)} \Omega_{a,b} G$, where $Z(G)$ denotes the centre of $G$.
\end{compactitem}
\end{definition}

\begin{definition}[Parametrised central groupoid of $G$]
Let $G$ be a path-connected topological group.
The {\em parametrised central groupoid of $G$} has
 objects the elements of $Z(G)$. 
 The set of all morphisms is  $[G, \Omega_Z G]$.
The set of morphisms from $a$ to $b$ is $[G, \Omega_{a,b} G]$.
Given objects $a,b,c \in Z(G)$ and maps $f_1 : G \rightarrow \Omega_{a,b} G$ and $f_2 : G \rightarrow \Omega_{b,c} G$, the composition $[f_2] \circ [f_1]$ is represented by $f_1 \ast f_2 :  G \rightarrow \Omega_{a,c} G$, defined by $(f_1 \ast f_2)(x) = f_1(x) \ast f_2(x)$, where $\ast$ denotes concatenation of paths.
\end{definition}

Let $G$ be a path-connected topological group and
recall that for a CW-complex $X$, isomorphism classes of principal $G$-bundles over $SX$ with structure group $G$ are in bijection with $[X, G]$. 

\begin{definition}[General $\gamma$-turning obstruction]  \label{d:gto}
Let $\gamma \in \Omega_Z G$. 
\begin{compactenum}[a)]
\item We define the map $\rho_{\gamma} : G \rightarrow \Omega_{\gamma} G$ by $\rho_{\gamma}(x) = (t \mapsto x\gamma(t)x^{-1})$.
\item For any CW-complex $X$ the {\em $\gamma$-turning obstruction map} is 
$\To_{\gamma} := (\rho_{\gamma})_* : [X, G] \rightarrow [X, \Omega_{\gamma} G]$.
\end{compactenum}
\end{definition}

Note that the image of $\rho_{\gamma}$ is contained in $\Omega_{\gamma(0),\gamma(1)} G$, because $\gamma(0),\gamma(1) \in Z(G)$ and in particular in the component $\Omega_{\gamma} G$, because $G$ is path-connected and $\rho_{\gamma}$ sends the identity element to $\gamma$. 
We let $0 \in [X, \Omega_\gamma G]$ denote the homotopy class of the constant map.
The proof of the following proposition is entirely analogous to the proof of Proposition \ref{p:to}.

\begin{proposition} \label{p:togamma}
Let $P$ be a $G$-bundle over $SX$ with clutching function $g \colon X \to G$. 
Then $P$ is $\gamma$-turnable if and only if $\To_{\gamma}([g]) = 0 \in [X, \Omega_{\gamma} G]$. 
\qed
\end{proposition}

\begin{definition}
Let $\To_G : \pi^Z(G) \rightarrow [G, \Omega_Z G]$ be defined by $\To_G([\gamma]) = [\rho_{\gamma}]$.
\end{definition}

This map is well-defined, because a path homotopy between $\gamma$ and $\gamma'$ determines a homotopy between the maps $\rho_{\gamma}$ and $\rho_{\gamma'}$. Since composition is defined in terms of concatenation both in $\pi^Z(G)$ and $[G, \Omega_Z G]$, it is a map of groupoids (with the identity map on the objects). With this notation $\To_{\gamma} = \To_G([\gamma])_* : [X, G] \rightarrow [X, \Omega_{\gamma} G] \subseteq [X, \Omega_Z G]$.

\begin{definition}
Let $PG = G / Z(G)$. Let $p_G : \pi^Z(G) \rightarrow \pi_1(PG)$ denote the map of groupoids induced by the projection $G \rightarrow PG$ (where $\pi_1(PG)$ is regarded as a groupoid on one object). 
\end{definition}

Every $\gamma \in \Omega_Z G$ determines a homeomorphism $p'_{\gamma} : \Omega_{\gamma} G \rightarrow \Omega_0 G$ which sends a path $\delta$ to the loop $t \mapsto \delta(t)\gamma(t)^{-1}$. If $[\gamma] = [\gamma']$, so that $\Omega_{\gamma} G = \Omega_{\gamma'} G$, then these homeomorphisms are homotopic, hence they induce a well-defined map $p'_{[\gamma]} : [G, \Omega_{\gamma} G] \rightarrow [G, \Omega_0 G]$. On the other hand, if $[\gamma] \neq [\gamma']$, then $\Omega_{\gamma} G$ and $\Omega_{\gamma'} G$ (and hence $[G, \Omega_{\gamma} G]$ and $[G, \Omega_{\gamma'} G]$) are disjoint, so $p'_G$ below is well-defined: 

\begin{definition}
Let $p'_G : [G, \Omega_Z G] \rightarrow [G, \Omega_0 G]$ denote the union of the maps $p'_{[\gamma]} : [G, \Omega_{\gamma} G] \rightarrow [G, \Omega_0 G]$.
\end{definition}

Since $\Omega_0 G$ is an H-space, $[G, \Omega_0 G]$ is a group, i.e.\ a groupoid on one object. For every pair $\gamma, \gamma'$ of composable paths in $\Omega_Z G$ the diagram
\[
\xymatrix{
\Omega_{\gamma} G \times \Omega_{\gamma'} G \ar[d]_-{p'_{\gamma} \times p'_{\gamma'}} \ar[r]^(0.575){\ast} & 
\Omega_{\gamma \ast \gamma'} G \ar[d]^-{p'_{\gamma \ast \gamma'}} \\
\Omega_0 G \times \Omega_0 G \ar[r]^(0.575){\ast} & \Omega_0 G
}
\]
commutes.  Hence $p'_G$ is a map of groupoids.

\begin{proposition} \label{p:desc}
Suppose that $G$ is a path-connected topological group and $Z(G)$ is discrete. Then $\To_G$ descends to a homomorphism $\widehat{\To}_G : \pi_1(PG) \rightarrow [G, \Omega_0 G]$ of abelian groups, i.e.\ there is a commutative diagram of groupoids: 
\[
\xymatrix{
\pi^Z(G) \ar[d]_-{p_G} \ar[r]^-{\To_G} & [G, \Omega_Z G] \ar[d]^-{p'_G} \\
\pi_1(PG) \ar[r]^-{\widehat{\To}_G} & [G, \Omega_0 G]
}
\]
\end{proposition}

\begin{proof}
First we define $\widehat{\To}_G$. Since $Z(G)$ is discrete, the projection $G \rightarrow PG$ is a covering, so a loop $\gamma \in \Omega PG$ can be lifted to a path $\tilde{\gamma} \in \Omega_Z G$ and we define $\widehat{\To}_G([\gamma]) = p'_G \circ \To_G([\tilde{\gamma}])$. If $\tilde{\gamma}'$ is another lift of $\gamma$, then $\tilde{\gamma}' = a\tilde{\gamma}$ for $a=\tilde{\gamma}'(0)\tilde{\gamma}(0)^{-1} \in Z(G)$. Then 
\[ 
(p'_{\tilde{\gamma}'} \circ \rho_{\tilde{\gamma}'}(x))(t) = x\tilde{\gamma}'(t)x^{-1}\tilde{\gamma}'(t)^{-1} = xa\tilde{\gamma}(t)x^{-1}\tilde{\gamma}(t)^{-1}a^{-1} = x\tilde{\gamma}(t)x^{-1}\tilde{\gamma}(t)^{-1} = (p'_{\tilde{\gamma}} \circ \rho_{\tilde{\gamma}}(x))(t) \text{\,,}
\]
so $p'_G \circ \To_G([\tilde{\gamma}']) = p'_G \circ \To_G([\tilde{\gamma}])$.
Therefore $\widehat{\To}_G([\gamma])$ does not depend on the choice of the lift $\tilde{\gamma}$. 
Since a homotopy of $\gamma$ can be lifted to a homotopy of $\tilde{\gamma}$, 
$\widehat{\To}_G([\gamma])$ is also independent of the choice of the representative $\gamma$ of the homotopy class $[\gamma]$. Therefore $\widehat{\To}_G$ is well-defined. By its construction, the diagram commutes. Finally, a lift of the concatenation of two loops in $\Omega PG$ is the concatenation of lifts of the loops, so $\widehat{\To}_G$ is a map of groupoids, i.e.\ a group homomorphism.
\end{proof}

In light of the above, it is useful define the normalised turning obstruction map in the general setting: 

\begin{definition}
Suppose that $G$ is a path-connected topological group and $\gamma \in \Omega_Z G$.
\begin{compactenum}[a)]
\item Let $\bar{\rho}_{\gamma} : G \rightarrow \Omega_0 G$ be defined by $\bar{\rho}_{\gamma}(x) = (t \mapsto x\gamma(t)x^{-1}\gamma(t)^{-1})$.
\item For any CW-complex $X$ let $\overline{\To}_{\gamma} = (\bar{\rho}_{\gamma})_* : [X, G] \rightarrow [X, \Omega_0 G]$.
\end{compactenum}
\end{definition}

\noindent
That is, $\bar{\rho}_{\gamma} = p'_{\gamma} \circ \rho_{\gamma}$ and hence 
$\overline{\To}_{\gamma} = (p'_{\gamma})_* \circ \To_{\gamma} \colon  [X, G] \rightarrow [X, \Omega_0 G]$. 
Since $p'_{\gamma}$ is a homeomorphism, $(p'_{\gamma})_* \colon [X, \Omega_{\gamma} G] \rightarrow [X, \Omega_0 G]$ is a bijection, so computing $\ol{\To}_{\gamma}$ is equivalent to computing $\To_{\gamma}$. In particular $\ol{\To}_\gamma$ and $\To_\gamma$ vanish for the same bundles. From our earlier arguments we have

\begin{proposition} \label{p:bto}
Suppose that $G$ is a path-connected topological group and $g \colon X \to G$ is a map.
\begin{compactenum}[a)]
\item If $\gamma, \gamma' \in \Omega_Z G$ are composable paths, then $\overline{\To}_{\gamma \ast \gamma'}([g]) = \overline{\To}_{\gamma}([g]) + \overline{\To}_{\gamma'}([g])$. 
\item For any $\gamma \in \Omega_Z G$ and $a \in Z(G)$ we have $\overline{\To}_{a\gamma}([g]) = \overline{\To}_{\gamma}([g])$.
\qed
\end{compactenum}
\end{proposition}

If $Z(G)$ is discrete, then we can use Proposition \ref{p:desc} to describe $\overline{\To}_{\gamma}([g])$, regarded as a two-variable function in $\gamma$ and $[g]$, in terms of a simpler function.
Hence we define 
\[
\widehat{\To}_{\cdot}(\cdot) \colon \pi_1(PG) \times [X, G] \rightarrow [X, \Omega_0 G] \text{\,,}
\quad \quad
\widehat{\To}_{[\gamma]}([g]) = \widehat{\To}_G([\gamma]) \circ [g] \text{\,.}
\]
Equivalently, $\widehat{\To}_{[\gamma]}([g]) = g^*(\widehat{\To}_G([\gamma])) = \widehat{\To}_G([\gamma])_*([g])$.

\begin{proposition} \label{p:whto}
Suppose that $G$ is a path-connected topological group and $Z(G)$ is discrete. 
For every CW-complex $X$ the maps $\overline{\To}_{\cdot}(\cdot)$ and $\widehat{\To}_{\cdot}(\cdot)$ satisfy
\begin{compactenum}[a)]
\item $\overline{\To}_{\gamma}([g]) = \widehat{\To}_{p_G([\gamma])}([g])$ for every $\gamma \in \Omega_Z G$ and $[g] \in [X, G]$;
\item $\widehat{\To}_{\cdot}([g]) : \pi_1(PG) \rightarrow [X, \Omega_0 G]$ is a homomorphism of abelian groups for every $[g] \in [X, G]$;
\item If $X$ is a suspension, then the map $\widehat{\To}_{[\gamma]} \colon [X, G] \rightarrow [X, \Omega_0 G]$ is a homomorphism of abelian groups for every $[\gamma] \in \pi_1(PG)$. 
\end{compactenum}
\end{proposition}

\begin{proof}
Part a) follows from the commutativity of the diagram in Proposition \ref{p:desc}. 

Part b) holds, because $\widehat{\To}_G$ is a homomorphism and the induced map $g^* \colon [G, \Omega_0 G] \rightarrow [X, \Omega_0 G]$ is a homomorphism for every $[g] \in [X, G]$, because $\Omega_0 G$ is an H-space.

Part c) holds, because $\widehat{\To}_{[\gamma]}$ is induced by a map of spaces.
\end{proof}

\begin{remark} \label{r:gato}
Just as with the turning obstruction for vector bundles,
we can take the forgetful adjoints of $\To_\gamma$ and $\ol{\To}_\gamma$
and define $\TO_\gamma : = \ad \circ \To_\gamma \colon [X, G] \to [SX, G]$
and $\ol{\TO}_\gamma : = \ad \circ \ol{\To}_\gamma \colon [X, G] \to [SX, G]$.
We leave the reader to formulate and verify the obvious generalisation of Lemma~\ref{l:to+bto}.
Then by Proposition~\ref{p:togamma},
a principal $G$-bundle $P \to SX$ with clutching function $g \colon X \to G$
is $\gamma$-turnable if and only if $\TO_\gamma([g]) = 0$,
equivalently if $\ol \TO_\gamma([g]) = 0$.
\end{remark}

We now deduce some consequences of Proposition~\ref{p:whto}.
Recall that $[\eta] \in \pi_1(SO_{2k})$ is the generator.

\begin{theorem} \label{thm:gpd-appl}
Let $\gamma \in \Omegapmone$, $X$ be a $CW$-complex and $E \to SX$ a rank-$2k$ vector bundle
with clutching function $g \colon X \to SO_{2k}$.
Then 
\begin{compactenum}[a)]
\item $2\ol{\To}_\eta([g]) = 0$;
\item $\ol{\To}_{\eta \ast \gamma}([g]) = \ol{\To}_{\eta}([g]) + \ol{\To}_\gamma([g])$;
\item If $k$ is even, then $2\ol{\To}_\gamma([g]) = 0$;
\item If $k$ is odd, then $2\ol{\To}_\gamma([g]) = \ol{\To}_\eta([g])$ and hence $4\ol{\To}_\gamma([g]) = 0$.
\end{compactenum}
\end{theorem}

\begin{proof}
Parts a) and b) are direct applications of Proposition~\ref{p:bto} a).

Parts c) and d) follow from Proposition~\ref{p:whto} a), Proposition~\ref{p:whto} b) and the fact that $\pi_1(PSO_{2k}) \cong (\Z/2)^2$ when $k$ is even and $\pi_1(PSO_{2k}) \cong \Z/4$ when $k > 1$ is odd; see Lemma \ref{l:pi_1PSO_2k}.
\end{proof}

\subsection{Samelson products and turning obstructions} \label{ss:SP_and_TO}
In this subsection we relate turning obstructions to Samelson products. The Samelson product is a classical operation in algebraic topology \cite{s53}, and Samelson products can be delicate to compute. First, we show that taking the Samelson product with some loop $[\gamma] \in \pi_1(G)$ coincides with the normalised turning obstruction map $\ol{\TO}_\gamma$ (after suitable identifications), see Lemma \ref{l:TO_and_SP}. Second, we show that turning obstructions in $G$ are determined by Samelson products in $PG$ (see Corollary \ref{c:to-from-sp}). As an application we determine some Samelson products based on our calculations of turning obstructions in Section \ref{ss:TOeta} (see Proposition \ref{p:sam-eta}). Finally, we show that our results on the $\eta$-turning obstruction have consequences for the high-dimensional topology of related gauge groups.

We start by recalling the definition of the Samelson product (in the special case when one of the operands is a loop). Assume that $X$ is connected with $x_0 \in X$ a base point and let $g \colon (X, x_0) \to (G, e)$ be a based map.
Let $\gamma \colon (S^1, 1) \to (G, e)$ be a map representing $[\gamma] \in \pi_1(G, e)$.
Then there is a well-defined map
\[ \mathrm{comm}_{g, \gamma} \colon X \wedge S^1 = \Sigma X \to G, 
\quad [x, t] \mapsto g(x)\gamma(t)g(x)^{-1}\gamma(t)^{-1} \text{\,.} \]
This construction gives rise to the {\em Samelson product} 
\[ [X, G]_* \times \pi_1(G, e) \to [\Sigma X, G]_*,
\quad ([g], [\gamma]) \mapsto \an{[g], [\gamma]} := [\mathrm{comm}_{g, \gamma} ] \text{\,.}\]
We can identify the set $[\Sigma X, G]_*$ with $[SX, G]$ via the forgetful map $[\Sigma X, G]_* \rightarrow [\Sigma X, G]$ and the map $[\Sigma X, G] \rightarrow [SX, G]$ induced by the collapse map $SX \rightarrow \Sigma X$ (which are both isomorphisms). The following lemma 
is a direct consequence of the definitions (where $\ol{\TO}_\gamma$ is defined in Remark~\ref{r:gato}).

\begin{lemma} \label{l:TO_and_SP}
For any $[\gamma] \in \pi_1(G, e)$ and $g \colon (X, x_0) \to (G, e)$ we have 
$\ol{\TO}_\gamma([g]) =\an{[g], [\gamma]} \in [SX, G]$. \qed
\end{lemma}

Lemma \ref{l:TO_and_SP} implies that certain Samelson products can be computed as a special case of turning obstructions. 
On the other hand, we next we show that turning obstructions can be computed from certain Samelson products.

\begin{definition}
Let $G$ be a path-connected topological group. We define the map $\Sp_G : \pi_1(G) \rightarrow [G, \Omega_0 G]$ by $\Sp_G([\gamma]) = [x \mapsto (t \mapsto x\gamma(t)x^{-1}\gamma(t)^{-1})]$.
\end{definition}

The map $\Sp_G$ is a homomorphism, because the group structure of $[G, \Omega_0 G]$ can be defined via concatenation in $\Omega_0 G$. This homomorphism encodes the Samelson product (similarly to how $\To_G$ encodes the turning obstruction): if $[\gamma] \in \pi_1(G)$ and $g : X \rightarrow G$, then $\an{[g], [\gamma]} \in [\Sigma X, G]_*$ is the adjoint of $\Sp_G([\gamma]) \circ [g] \in [X, \Omega_0 G]_*$.

Let $\pi : G \rightarrow PG$ denote the projection.

\begin{proposition} \label{p:sp-to}
Suppose that $G$ is a path-connected topological group and $Z(G)$ is discrete. Then there is a commutative diagram of groups
\[
\xymatrix{
\pi_1(PG) \ar[r]^-{\widehat{\To}_G} \ar[d]_-{\Sp_{PG}} & [G, \Omega_0 G] \ar[d]^-{(\Omega_0 \pi)_*} \\
[PG, \Omega_0 PG] \ar[r]^-{\pi^*} & [G, \Omega_0 PG]
}
\]
where $(\Omega_0 \pi)_*$ is an isomorphism.
\end{proposition}

\begin{proof}
Let $\gamma \in \Omega PG$. Since $Z(G)$ is discrete, the projection $\pi : G \rightarrow PG$ is a covering, so $\gamma$ can be lifted to a path $\tilde{\gamma} \in \Omega_Z G$. By definition we have $\widehat{\To}_G([\gamma]) = [x \mapsto (t \mapsto x\tilde{\gamma}(t)x^{-1}\tilde{\gamma}(t)^{-1})]$. Its image in $[G, \Omega_0 PG]$ is $[x \mapsto (t \mapsto \pi(x\tilde{\gamma}(t)x^{-1}\tilde{\gamma}(t)^{-1}))] = [x \mapsto (t \mapsto \pi(x)\gamma(t)\pi(x)^{-1}\gamma(t)^{-1})]$, using that $\pi \circ \tilde{\gamma} = \gamma$. The image of $\Sp_G([\gamma]) = [y \mapsto (t \mapsto y\gamma(t)y^{-1}\gamma(t)^{-1})]$ in $[G, \Omega_0 PG]$ is also $[x \mapsto (t \mapsto \pi(x)\gamma(t)\pi(x)^{-1}\gamma(t)^{-1})]$, therefore the diagram commutes. 

Since $\pi : G \rightarrow PG$ is a covering, every nullhomotopic loop in $PG$ can be lifted to a nullhomotopic loop in $G$, hence $\Omega_0 \pi : \Omega_0 G \rightarrow \Omega_0 PG$ is a homeomorphism. Moreover, this homeomorphism respects the H-space structures, therefore $(\Omega_0 \pi)_* : [G, \Omega_0 G] \rightarrow [G, \Omega_0 PG]$ is an isomorphism. 
\end{proof}

Recall that by Proposition \ref{p:whto} the turning obstruction map $\overline{\To}_{\gamma}$ can be computed from $\widehat{\To}_G$, namely $\overline{\To}_{\gamma}([g]) = \widehat{\To}_G(p_G([\gamma])) \circ [g] \in [X, \Omega_0 G]$ for every $\gamma \in \Omega_Z G$ and $g : X \rightarrow G$. By Proposition \ref{p:sp-to} we have $\widehat{\To}_G(p_G([\gamma])) = [(\Omega_0 \pi)^{-1}] \circ \Sp_{PG}(p_G([\gamma])) \circ [\pi]$, hence $\overline{\To}_{\gamma}([g]) = [(\Omega_0 \pi)^{-1}] \circ \Sp_{PG}(p_G([\gamma])) \circ [\pi \circ g]$. This shows that $\overline{\To}_{\gamma}$ is determined by $\Sp_{PG}(p_G([\gamma]))$, that is, turning obstructions in $G$ are determined by Samelson products in $PG$. We can also express this in terms of the adjointed versions:

\begin{corollary} \label{c:to-from-sp}
Suppose that $G$ is a path-connected topological group and $Z(G)$ is discrete. Let $\gamma \in \Omega_Z G$ and $g : X \rightarrow G$. Then $\overline{\TO}_{\gamma}([g]) = (\pi_*)^{-1}(\an{[\pi \circ g], p_G([\gamma])}) \in [SX, G]$, where $(\pi_*)^{-1}$ is the inverse of the isomorphism $\pi_* : [SX, G] \rightarrow [SX, PG]$.
\qed
\end{corollary}

\begin{remark}
If $X$ is simply-connected, then $\pi_* : [X, G] \rightarrow [X, PG]$ is also an isomorphism, which allows us to take the reverse point of view and compute Samelson products in $PG$ from turning obstructions in $G$: Suppose that $\gamma \in \pi_1(PG)$ and $g : X \rightarrow PG$, then $\an{[g], [\gamma]} = \pi_*(\overline{\TO}_{\tilde{\gamma}}((\pi_*)^{-1}([g]))) \in [SX, PG]$, where $\tilde{\gamma} \in \Omega_Z G$ is a lift of $\gamma$. 
\end{remark}

In the next section we will compute various turning obstructions (see Theorem~\ref{t:to_S2k}). By Lemma~\ref{l:TO_and_SP} those results give a variety of information about Samelson products $\an{[g], \eta}$ for $\eta \in \pi_1(SO_{2k})$ the generator, for example we get the following proposition. Recall that $\tau_{2k} \in \pi_{2k-1}(SO_{2k})$ is the homotopy class of a clutching function of the tangent bundle of $S^{2k}$, and for $m > 2$ let $\eta_{m} \colon S^{m+1} \to S^{m}$ be essential.

\begin{proposition} \label{p:sam-eta}
The Samelson product $\an{\tau_{2k}, \eta} \in \pi_{2k}(SO_{2k})$ is given as follows:
\begin{compactenum}[a)]
\item If $k = 2j{+}1$ is odd, $\an{\tau_{4j+2}, \eta} = 0$;
\item If $k = 2j$ is even, $\an{\tau_{4j}, \eta} = \tau_{4j}\eta_{4j-1} \neq 0$. \qed
\end{compactenum}
\end{proposition}

\begin{remark}
Proposition \ref{p:sam-eta} can be viewed as an extension of an odd-primary theorem of Hamanaka and Kono \cite[Theorem A]{hk07} to the prime $2$.
\end{remark}

As another application of Lemma~\ref{l:TO_and_SP}, we consider the situation where $\eta$ is not the
turning datum in a turning problem, but instead the clutching function of a bundle.
Let $E^{2k}_\eta \to S^2$ be a fixed non-trivial oriented rank-$2k$ bundle over $S^2$.
Then $\Fr(E)$ is a non-trivial principal $SO_{2k}$-bundle over $S^2$,
we write $\GG^{2k}_\eta$ for the gauge group of $\Fr(E^{2k}_\eta)$
and consider the fibration sequence \eqref{eq:LES1}
for $\GG^{2k}_\eta$, which we write as $\GG^{2k}_{\eta,0} \to \GG^{2k}_\eta \to SO_{2k}$.
As discussed in Section~\ref{ss:gg-eta}, $\GG^{2k}_{\eta, 0} \cong \Map((S^2, \ast), (SO_{2k}, \mathrm{Id}))$.
Hence there is a natural isomorphism $\pi_i(\GG^{2k}_{\eta, 0}) \cong \pi_{i+2}(SO_{2k})$ and
by a theorem of Wockel \cite[Theorem 2.3]{w07} the boundary map
\[ \del^{2k}_\eta \colon \pi_i(SO_{2k}) \to \pi_{i-1}(\GG^{2k}_{\eta, 0}) = \pi_{i+1}(SO_{2k})\]
in the associated long exact sequence is given by $\del^{2k}_\eta([g]) = -\an{[g], \eta}$, for all $[g] \in \pi_i(SO_{2k})$.
Combining Lemma~\ref{l:TO_and_SP} and Theorem~\ref{t:to_S2k} therefore gives 
information about the map $\del^{2k}_\eta$.  
In particular, for $\tau_{2k}$ and $\eta_{4j-1}$ as in Proposition \ref{p:sam-eta}
we have

\begin{proposition} \label{p:GGE_eta}
The boundary map $\del^{2k}_\eta \colon \pi_{2k-1}(SO_{2k}) \to \pi_{2k}(SO_{2k})$ satisfies
$\del^{4j}_\eta(\tau_{4j}) = \tau_{4j}\eta_{4j-1} \neq 0$ and $\del^{4j+2}_\eta(\tau_{4j+2}) = 0$.
\qed
\end{proposition}

\begin{remark}
If we let $E^\infty_\eta$ denote the stabilisation of the $E^{2k}_\eta$, then its frame bundle $\Fr(E^\infty_\eta)$
is a non-trivial principal $SO$-bundle over $S^2$
and we let $\GG^\infty_\eta$ denote the gauge group of $E^\infty_\eta$.
Since the stable group $SO$ is a homotopy abelian $H$-space,
it follows that there is a weak 
homotopy equivalence
\begin{equation*} \label{eq:GGinfty}
\GG^\infty_\eta \simeq \Map(S^2, SO) \cong \Map_*(S^2, SO) \times SO.
\end{equation*} 
By comparing the homotopy long exact sequences of the fibrations
$\GG^{4j}_{\eta, 0} \to \GG^{4j}_\eta \to SO_{4j}$ 
and $\GG^{\infty}_{\eta, 0} \to \GG^\infty_\eta \to SO$,
where $\GG^{\infty}_{\eta, 0} \subset \GG^\infty_\eta$ is the group of gauge transformations
which are the identity in the fibre over the base-point, 
we see that $\del^{4j}_\eta \colon \pi_{i}(SO_{4j}) \to \pi_{i+1}(SO_{4j})$ is zero for $i < 4j{-}2$.
When $i = 4j{-}2$, the domain of $\del^{4j}_\eta$ is $\pi_{4j-2}(SO_{4j}) \cong \pi_{4j-2}(SO) = 0$,
and so $\del^{4j}_\eta$ vanishes for $i \leq 4j{-}2$.  
Hence Proposition~\ref{p:GGE_eta} shows that the first possibly non-zero boundary 
map in the homotopy long exact sequence of $\GG^{4j}_{\eta, 0} \to \GG^{4j}_\eta \to SO_{4j}$
is in fact non-zero.
\end{remark}

\section{Turning rank-$2k$-bundles over the $2k$-sphere}  \label{s:2k_over_S2k}
In this section we determine the turning obstructions for oriented rank-$2k$-bundles over the
$2k$-sphere.  
To state our results, it will be convenient to use the notation $\TO_+ :=  \TO_\beta$ and
$\TO_- := \TO_{\bar \beta}$ and when 
we wish to discuss these obstructions together, we will write $\TO_\pm$.
We also define the adjointed $\eta$-turning obstruction 
$\TO_\eta := \ad \circ \To_\eta \colon \pi_{2k-1}(SO_{2k}) \to \pi_{2k}(SO_{2k})$.
With this notation, the goal of this section is to compute the homomorphisms
\[ \TO_\pm \colon \pi_{2k-1}(SO_{2k}) \to \pi_{2k}(SO_{2k})
\quad \text{and} \quad
\TO_\eta \colon \pi_{2k-1}(SO_{2k}) \to \pi_{2k}(SO_{2k}).
\]
Thus, if $E \to S^{2k}$ is a rank-$2k$ vector bundle with clutching function
$g \colon S^{2k-1} \to SO_{2k}$, 
then $E$ is positive-turnable if and only if $\TO_+([g]) = 0$, $E$ is negative-turnable
if and only if $\TO_-([g]) = 0$ and $E$ is $\eta$-turnable if and only if $\TO_\eta([g]) = 0$.

In order to state the computations of $\TO_\pm$ and $\TO_\eta$ we record some facts we need about the
source and target groups of these homomorphisms,
which can be found in \cite{k60}.  
We also introduce notation for generators of these groups.
Recall that $\mathrm{e}(\xi) = \mathrm{e}(E_\xi) \in \Z$ is the Euler class and 
that $\mathrm{e}(\xi)$ is even unless $k = 2, 4$, \cite[Theorem p.87]{bm58}.
Let $\tau_{2k} \in \pi_{2k-1}(SO_{2k})$ denote the homotopy class of the clutching function
of $TS^{2k}$.
There is an isomorphism
\[ \pi_{2k-1}(SO_{2k}) \cong \Z(\tau_{2k}) \oplus C(\sigma_{2k}),\]
where $C(\sigma_{2k})$ is a cyclic group isomorphic to $\pi_{2k-1}(SO)$ and $S(\sigma_{2k}) \in \pi_{2k-1}(SO)$
is a generator.  When $k = 2, 4$, we take $\mathrm{e}(\sigma_{2k}) = 1$ and when $k=1$ we assume that
$\sigma_4$ admits a complex structure; see Definition~\ref{d:sigma4} and Theorem~\ref{t:S4}. 
When $k \neq 2, 4$, by Lemma~\ref{l:sigma2k} below, we assume that $C(\sigma_{2k}) = S(\pi_{2k-1}(SO_{2k-2}))$;
in particular $\mathrm{e}(\sigma_{2k}) = 0$.

There are isomorphisms
\[ \pi_{2k}(SO_{2k}) \cong
\begin{cases}
0 & k = 1, 3, \\
\Z/4 & \text{$k \geq 5$ is odd}, \\
(\Z/2)^2 & \text{$k \equiv 2$ mod $4$},\\
(\Z/2)^3 & \text{$k \equiv 0$ mod $4$}.
\end{cases}
\]
When $k$ is odd, we let $\zeta \in \pi_{2k}(SO_{2k})$ be a generator and note that $\pi_{2k}(SO) = 0$.
When $k$ is even, the stabilisation homomorphism $S \colon \pi_{2k}(SO_{2k}) \to \pi_{2k}(SO)$ is split 
onto, where $\pi_{2k}(SO) = 0$ if $k \equiv 2$ mod~$4$ and $\Z/2$ if $k \equiv 0$ mod~$4$.  
Moreover, for all $k$ there is a short exact sequence
\begin{equation} \label{eq:pi2kSO2k}
0 \to \Z/2(\tau_{2k}\eta_{2k-1}) \to \pi_{2k}(SO_{2k}) \xra{\mathrm{ev_*} \oplus S} \pi_{2k}(S^{2k-1}) 
\oplus \pi_{2k}(SO) \to 0,
\end{equation}
where $\mathrm{ev} \colon SO_{2k} \to S^{2k-1}$ is given by evaluation at a point in $S^{2k-1}$
and $\eta_{2k-1} \colon S^{2k} \to S^{2k-1}$ is essential.
The sequence \eqref{eq:pi2kSO2k} is non-split when $k \geq 5$ is odd and splits when $k$ is even.

\begin{theorem} \label{t:to_S2k}
The turning obstructions $\TO_\pm \colon \pi_{2k-1}(SO_{2k}) \to \pi_{2k}(SO_{2k})$ satisfy the following:
\begin{compactenum}[a)]
\item If $k$ is odd then, $\TO_\pm(\xi)  = \mathrm{e}(\xi)\zeta$;
\item If $k \equiv 2$ mod $4$ then $\mathrm{ev}_*(\TO_\pm(\tau_{2k})) = 1$, $\TO_+(\sigma_{2k}) = 0$ and $\TO_-(\sigma_{2k}) = \mathrm{e}(\sigma_{2k})\TO_-(\tau_{2k})$;
\item If $k \equiv 0$ mod $4$ then $\mathrm{ev}_*(\TO_\pm(\tau_{2k})) = 1$, $S(\TO_\pm(\tau_{2k})) = 0$ and $S(\TO_\pm(\sigma_{2k})) = 1$; in particular $\TO_\pm \otimes \id_{\Z/2}$ is injective.
\end{compactenum}
In particular, if $k$ is odd then $\TO_\eta= \TO_+ - \TO_- = 0$. If $k = 2j$ is even then $\TO_\eta$ satisfies the following:
\begin{compactenum}[a)] \setcounter{enumi}{3}
\item $\TO_\eta(\tau_{4j})= \tau_{4j}\eta_{4j-1} \neq 0$;
\item If $j = 1$ or $2$, then $\mathrm{ev}_*(\TO_\eta(\sigma_{4j})) =  1$ and $\TO_\eta \otimes \id_{\Z/2}$ is injective;
\item If $j \geq 3$, then $\TO_\eta(\sigma_{4j}) = 0$.
\end{compactenum}
\end{theorem}

\begin{remark} \label{r:schirality}
Theorem~\ref{t:to_S2k} shows that unless $k = 2$, for all $[g] \in \pi_{2k-1}(SO_{2k})$
we have $\TO_+([g]) = 0$ if and only if $\TO_-([g]) = 0$.
Hence for $k \neq 2$ rank-$2k$ bundles $E \to S^{2k}$ are either bi-turnable or not turnable
and so these bundles are not strongly chiral.
On the other hand, when $k = 2$, a bundle $E \to S^4$ is strongly chiral if and only if $\mathrm{e}(E)$ is odd.
\end{remark}

The remainder of this section is devoted to the proof of Theorem~\ref{t:to_S2k}.
In Section~\ref{ss:tau_2k} we consider the turnability of the tangent bundle
of the $2k$-sphere, which is an essential input to the proof.
In Section~\ref{ss:TOeta} we consider $\TO_{\eta} = \TO_+ - \TO_-$ 
and prove parts d)-f) of Theorem~\ref{t:to_S2k}.
In Section~\ref{ss:S4} 
we consider the exceptional case of the
$4$-sphere 
and in Section~\ref{ss:TO_S2k}
we assemble the previous work to prove parts a)-c) of Theorem~\ref{t:to_S2k}.

\subsection{Turning the tangent bundle of the $2k$-sphere}  \label{ss:tau_2k}
Let $TS^n$ denote the tangent bundle of the $n$-sphere.
We fix the standard orientation on the $n$-sphere and this orients $TS^n$.
In \cite{k47} Kirchoff proved that if $TS^{2k}$ admits a complex structure then
$TS^{2k+1}$ is trivial.  
Later \cite{bm58}, it was proven that $TS^{2k+1}$ is trivial if and only if $2k{+}1 = 1, 3$ or $7$.
Since elementary calculations show that $TS^2$ and $TS^6$ admit complex structures,
Kirchoff's theorem implies that $TS^{2k}$ admits a complex structure if and only 
if $TS^{2k+1}$ is trivial.
Here we prove a strengthening of Kirchoff's theorem, which only assumes that
$TS^{2k}$ is turnable.  

\begin{theorem}(Kirchoff's theorem for turnings) \label{t:SKirchoff}
If $TS^{2k}$ is turnable then $TS^{2k+1}$ is trivial. 
\end{theorem}

\begin{corollary} \label{c:SKirchoff}
$TS^{2k}$ is turnable if and only if it admits a complex structure if and only 
if $2k = 2$ or $6$. \qed
\end{corollary}

\begin{proof}[Proof of Theorem~\ref{t:SKirchoff}]
We first recall the following well-known definition of a clutching
function $c_m$ for $TS^m$; see~\cite[Ch.\ 8, Corollary 9.9]{h94}.
  Given $x \in S^{m-1}$, write $\R^m = \an{x} \oplus \an{x}^\perp$ 
as the sum of the line spanned by $x$ and its orthogonal complement
and write $v \in \R^m$ as $v = (w, y)$ where $w \in \an{x}$ and $y \in \an{x}^\perp$.
Let $c_m: S^{m-1} \to O_m$ be the function which assigns to $x \in S^{m-1}$ 
the reflection to the hyperplane orthogonal to $x$:
\[ c_m(x): \R^m \to \R^m, 
\quad
(w, y) \mapsto (-w, y) \]

Suppose that $TS^{2k}$ is turnable. 
We will show that the clutching function 
$c_{2k+1}:S^{2k} \to O_{2k+1}$ is null-homotopic, proving that 
$TS^{2k+1}$ is trivial. 
Using the notation above, we see that 
\[ TS^{2k} = \{((0, y), x) \mid y \in \an{x}^\perp \} \subset \R^{2k+1} \times S^{2k}.\]
Since $TS^{2k}$ is turnable, there exists a turning $\alpha_t$ on $TS^{2k}$ with
$\alpha_0 = \idbb$ and $\alpha_1 = -\idbb$.
We use $\alpha_t$ to define the following homotopy of automorphisms of the
trivial bundle
\[ H : (\R^{2k+1} \times S^{2k}) \times I \to \R^{2k+1} \times S^{2k},
\quad \bigl( \bigl( (w, y), x \bigr), t \bigr) \mapsto \bigl((-w, \alpha_t(y)), x \bigr).
 \]
We see that $H_0 = c_{2k+1}$ and $H_1 = -\idbb$.  
Hence $H$ is the required homotopy of clutching functions from $c_{2k+1}$ to a
constant map.
\end{proof}

\subsection{The turning obstruction $\TO_{\eta}$}  \label{ss:TOeta}
In this subsection we compute $\TO_{\eta} \colon \pi_{2k-1}(SO_{2k}) \to
\pi_{2k}(SO_{2k})$ and prove parts d)-f) of Theorem~\ref{t:to_S2k}.
The computation of $\TO_\eta$ requires some preliminaries.

For any $m \geq 2$, let $V_{m, 2}$ be the Stiefel manifold of ordered pairs of orthonormal
vectors in $\R^{m}$.  
Given $\ul v = (v_1, v_2) \in V_{m, 2}$ we define $V = \an{v_1, v_2}$ and write 
$x \in \R^m$ as $x = (v, w)$ where $v \in V$ and $w \in V^\perp$.
The isomorphism $\C \rightarrow V$, $1 \mapsto v_1$, $i \mapsto v_2$ defines a complex structure on $V$. 
We define $\gamma_{\ul v}$ in $\Omega SO_m$ by
\[ \gamma_{\ul v}(t)(v, w) = (e^{2 \pi it}v, w) \]
and we define the map
\[ L = L_m \colon V_{m, 2} \to \Omega SO_m,
\quad \ul v \mapsto \gamma_{\ul v}.\]
Next we consider the canonical projection $p \colon SO_m \to V_{m, 2}$ and the composition
\[ L \circ p \colon SO_m \to \Omega SO_m.\]
It is clear from the definitions that $L \circ p$ is the map $\rho_{\eta}$ of 
Definition~\ref{d:gto} a), so after the identification $\pi_{m-1}(\Omega SO_m) = \pi_m(SO_m)$ we obtain the following

\begin{lemma} \label{l:L_circ_p}
For all $[g] \in \pi_{m-1}(SO_m)$, $(L \circ p)_*([g]) = \TO_\eta([g])$. \qed
\end{lemma}

\begin{remark} \label{r:SPmethod}
Combining Lemmas~\ref{l:L_circ_p}, \ref{l:to+bto} and~\ref{l:TO_and_SP}, we get 
$(L \circ p)_*([g]) = \TO_\eta([g]) = \ol{\TO}_\eta([g]) = \an{[g], \eta}$
and this equation can be generalised to give a method for computing
similar Samelson products as follows.

For $2 \leq i \leq m$, 
let $V_{m, i}$ denote the Stiefel manifold of mutually orthonormal ordered $i$-tuples $\ul v = (v_1, \dots, v_i)$ 
of vectors in $\R^m$, set $V = \an{v_1, \dots, v_i} = \R^i$ and write
$x \in \R^m$ as $x = (v, w)$, where $v \in V$ and $w \in V^\perp$.
Then given any map $\alpha \colon S^{i-1}\to SO_{i}$ we define
$\alpha_{\ul v} \in \Omega^{i-1}SO_m$, the $(i{-}1)$-fold based loop space of $SO_m$, by
\[ \alpha_{\ul v}(s)(v, w) = (\alpha(s) v, w), \]
for all $s \in S^{i-1}$ and $(v, w) \in \R^m$.  Allowing $\ul v$ to vary, we obtain the map
\[ L(\alpha) \colon V_{m, i} \to \Omega^{i-1}SO_m, \quad
 \ul v \mapsto \alpha_{\ul v}, \]
and note that $L \colon V_{m, 2} \to \Omega SO_m$ above is $L(\alpha)$ for the special case
of $\alpha \colon S^1 \to SO_2 = U(1), t \mapsto e^{2 \pi i t}$.
If $\iota \colon SO_i \to SO_m$ denotes the standard inclusion, and $p \colon SO_m \to V_{m, i}$
the standard projection, then after the identification $\pi_j(\Omega^{i-1}SO_m) = \pi_{i+j-1}(SO_m)$,
a higher dimensional version of Lemma~\ref{l:to+bto} leads to
the equation
\[ (L(\alpha) \circ p)_*([g]) = \An{[g], \iota_*([\alpha])} \]
for all $[g] \in \pi_j(SO_m)$.
\end{remark}

Now we consider the case $m = 2k$ and 
the induced homomorphisms 
$p_* \colon \pi_{2k-1}(SO_{2k}) \to \pi_{2k-1}(V_{2k, 2})$
and $L_* \colon \pi_{2k-1}(V_{2k, 2}) \to \pi_{2k}(SO_{2k})$.
Let $\mathrm{ev} \colon SO_{2k} \to S^{2k-1}$ be the map defined by
evaluation at a point in $S^{2k-1}$.

\begin{definition}
Define $a_k \in \Z$ by $a_k := 1$ if $k$ is even and $a_k := 2$ if $k$ is odd.
\end{definition}

\begin{lemma} \label{l:p_*L_*}
For any isomorphism $\pi_{2k-1}(V_{2k, 2}) \to \Z/2 \oplus \Z$
we have:
\begin{compactenum}[a)]
\item $p_*(\tau_{2k}) = (\rho_2(a_k), 2)$;
\item $L_*(1, 0) = \tau_{2k}\eta_{2k-1}$;
\item $\mathrm{ev}_*(L_*(0, 1)) = \rho_2(a_k)$.
\end{compactenum}
\end{lemma}

\begin{proof}
a) The sequence $\pi_{2k-1}(S^{2k-2}) \to \pi_{2k-1}(V_{2k, 2}) \to \pi_{2k-1}(S^{2k-1})$ 
is split short exact, with $p_*(\tau_{2k})$ mapping to $2 \in \pi_{2k-1}(S^{2k-1}) = \Z$.
It follows that the map $f = p \circ \tau_{2k} \colon S^{2k-1} \to V_{2k,2}$ vanishes on mod $2$
cohomology.  Also for $x \in H^{2k-2}(V_{2k,2}; \Z/2)$ the generator, $Sq^2(x) \in H^{2k}(V_{2k, 2}; \Z/2) = 0$
and so the functional Steenrod square $Sq^2_x(g)$ is defined for all maps $g \colon S^{2k-1} \to V_{2k, 2}$ which vanish on mod $2$ cohomology.  Moreover, $g = 2g'$ for some map $g' \colon S^{2k-1} \to V_{2k, 2}$
if and only if $Sq^2_x(g) = 0$.

Now the map $\tau_{2k} \colon S^{2k-1} \to SO_{2k}$ factors over the double covering
$q \colon S^{2k-1} \to \R P^{2k-1}$ and a map $\tau_{2k}' \colon \R P^{2k-1} \to SO_{2k-1}$.
Since $q$ vanishes on mod $2$ cohomology and $Sq^2(t^{2k-2}) = 0$, for
$t \in H^1(\R P^{2k-1}; \Z/2)$ a generator, it follows that the functional Steenrod square
$Sq^2_{t^{2k-2}}$ is defined on $q$.  We consider the composition
\[ S^{2k-1} \xra{~q~} \R P^{2k-1} \xra{~\tau'_{2k}~} SO_{2k} \xra{~p~} V_{2k, 2}. \]
A degree argument shows that the map $p \circ \tau'_{2k} \colon \R P^{2k-1} \to V_{2k, 2}$
satisfies $(p \circ \tau'_{2k})^*(x) = t^{2k-2}$ and naturality of functional Steenrod squares
gives that
\[ Sq^2_x(p \circ \tau_{2k}) = Sq^2_{t^{2k-2}}(q).\]
But $q$ is the attaching map of the top cell of $\R P^{2k}$ and so
$Sq^2_{t^{2k-2}}(q)= Sq^2(t^{4k}) = \rho_2(a_k)$.
Hence we have 
$Sq^2_x(p \circ \tau_{2k}) = \rho_2(a_k)$ and so $p_*(\tau_{2k}) = (\rho_2(a_k), 2) \in \pi_{2k-1}(V_{2k, 2})$.

b) Let $\iota_{2k-2} \colon S^{2k-2} \to V_{2k, 2}$ be the inclusion of a fibre of the projection
$V_{2k, 2} \to S^{2k-1}$.  Then we have $(1, 0) = [\iota_{2k-2} \circ \eta_{2k-2}] \in \pi_{2k-1}(V_{2k, 2})$.
Hence it suffices to prove that $L_*([\iota_{2k-2}])= \tau_{2k} \in \pi_{2k-1}(SO_{2k})$.
Now a degree argument shows that 
$\mathrm{ev}_*(L_*([\iota_{2k-2}])) = 1-(-1)^{2k-2} = 2 \in \pi_{2k-1}(S^{2k-1}) = \Z$
and we consider the following commutative diagram:
\[
\xymatrix{
V_{2k ,2} \ar[d] \ar[r]^(0.45){L_{2k}} &
\Omega SO_{2k} \ar[d]^S \\
V_{2k+2, 2} \ar[r]^(0.425){L_{2k+2}} &
\Omega SO_{2k+2}
}
\]
Since $\pi_{2k-2}(V_{2k+2,2}) = 0$, we see that $S(L_*([\iota_{2k-2}])) = 0$.
Hence $L_*([\iota_{2k-2}])$ is the clutching function of a stably trivial bundle
with Euler class $2$, so $L_*([\iota_{2k-2}]) = \tau_{2k}$, as required.

c) The standard complex structure on $\R^{2k} = \C^k$ defines a section 
$s \colon S^{2k-1} \rightarrow V_{2k, 2}$ 
of the projection $V_{2k, 2} \to S^{2k-1}$ by $s(v)=(v, iv)$.
Taking induced maps on $\pi_{2k-1}$ gives a splitting $\pi_{2k-1}(V_{2k, 2}) \cong \Z/2 \oplus \Z$,
where $[s] = (0, 1)$ and $[\iota_{2k-2} \circ \eta_{2k-2}] = (1, 0)$.
Since $\mathrm{ev}_*(L_*(1, 0)) = 0$, it suffices to prove $\mathrm{ev}_*(L_*([s])) = \rho_2(a_k)$.

It is clear that $L \circ s$ factors as the composition of the Hopf map $H \colon S^{2k-1} \to \C P^{k-1}$ and a map $L' \colon \C P^{k-1} \to \Omega SO_{2k}$. Another degree argument shows that the adjoint of $\Omega(\mathrm{ev}) \circ L' :\C P^{k-1} \to \Omega S^{2k-1}$ has degree one and so the homotopy class of $S^{2k-1} \to \Omega S^{2k-1}$ is determined by the functional Steenrod square $Sq^2_{z^{k-1}}(H)$, where $z \in H^2(\C P^{k-1}; \Z/2)$ is the generator. Since $H$ is the attaching map of the top cell of $\C P^k$, we have
\[Sq^2_{z^{k-1}}(H) = Sq^2(z^{k-1}) = \rho_2(a_k) \in \Z/2 \cong H^{2k}(\C P^k; \Z/2) \text{\,,} \]
which completes the proof of part c).
\end{proof}

\begin{lemma} \label{l:sigma2k}
If $k \neq 2, 4$, we may chose $\sigma_{2k} \in S(\pi_{2k-1}(SO_{2k-2})) \subset \pi_{2k-1}(SO_{2k})$.
\end{lemma}

\begin{proof}
If $k \neq 2, 4$, then $\sigma_{2k} \in \pi_{2k-1}(SO_{2k})$ is such that $S(\sigma_{2k})$ generates
$\pi_{2k-1}(SO)$ and $\mathrm{e}(\sigma_{2k}) = 0$.
The map $\pi_{2k-2}(SO_{2k}) \to \pi_{2k-1}(SO)$ is onto for $k \geq 7$ by \cite{bm64},
which proves the lemma when $k \geq 7$.
For the remaining cases $k \in \{3, 5, 6\}$. If $k = 3, 6$, then $\pi_{2k}(SO) = 0$, so in these cases 
$\sigma_{2k} = 0$.
When $k = 5$ we consider part of the homotopy long exact sequence of the fibration $SO_8 \to SO_9 \to S^8$. 
\[ \pi_{9}(SO_8) \to \pi_9(SO_9) \to \pi_9(S^8) \to \pi_8(SO_8) \to \pi_8(SO_9)\]
By results of Kervaire \cite{k60}, $\pi_8(SO_8) \cong (\Z/2)^3$ and $\pi_8(SO_9) \cong (\Z/2)^2$.
Since $\pi_9(S^8) \cong \Z/2$, the boundary map $\pi_9(S^8) \to \pi_8(SO_8)$ is injective
and so $\pi_{9}(SO_8) \to \pi_9(SO_9)$ is onto.
Since $\mathrm{e}(\sigma_{10}) = 0$, $\sigma_{10} \in S(\pi_9(SO_9))$
and so $\sigma_{10} \in S(\pi_9(SO_8))$.
\end{proof}

\begin{lemma} \label{l:tospn-kodd}
If $k$ is odd, then $\TO_\eta([g])= 0$ for all $[g] \in \pi_{2k-1}(SO_{2k})$.
\end{lemma}

\begin{proof}
By Lemma~\ref{l:p_*L_*} a), $p_*(\tau_{2k}) = (0, 2)$
and by Lemma~\ref{l:p_*L_*} c), $L_*(0, 2) = 0$.  
Hence by Lemma~\ref{l:L_circ_p},
$\TO_\eta(\tau_{2k})= (L \circ p)_*(\tau_{2k}) = 0$.
If $k \equiv 1$~mod~$4$, then by Lemma~\ref{l:sigma2k},
$\sigma_{2k} \in S(\pi_{2k}(SO_{2k-2}))$ and so 
$\TO_\eta(\sigma_{2k}) = 0$.
\end{proof}

\begin{proof}[Proof of parts d)-f) of Theorem ~\ref{t:to_S2k}]
d) If $k = 2j$ is even, then $p_*(\tau_{4j}) = (1, 2)$ by by Lemma~\ref{l:p_*L_*} a)
and since $\pi_{4j}(SO_{4j})$ is a $2$-torsion group, $L_*(1, 2) = L_*(1, 0)$.
By  Lemma~\ref{l:p_*L_*} b), $L_*(1, 0) = \tau_{4j}\eta_{4j-1}$ and so
$\TO_\eta(\tau_{4j})= (L \circ p)_*(\tau_{4j}) = \tau_{4j}\eta_{4j-1}$.

e) If $j = 1, 2$, then $\mathrm{e}(\sigma_{4j}) = 1$ and so $p_*(\sigma_{4j}) = (\epsilon, 1)$.
Since $\mathrm{ev}_*(L_*(1, 0)) = 0$, Lemma~\ref{l:p_*L_*} c) ensures that
 $\mathrm{ev}_*(\TO_\eta(\sigma_{4j})) = \mathrm{ev}_*((\epsilon, 1)) = 1$.

f) If $j \geq 3$ then by by Lemma~\ref{l:sigma2k}$, \sigma_{4j} \in S(\pi_{4j}(SO_{4j-2}))$ and so 
$\TO_\eta(\sigma_{4j}) = 0$.
\end{proof}

\subsection{Rank-$4$ bundles over the $4$-sphere} \label{ss:S4}
The set of isomorphism classes of rank-$4$ bundles over $S^4$ is in bijection with $\pi_3(SO_4) \cong \Z \oplus \Z$. 
We recall the canonical double covering
\[ q \colon S^3 \times S^3 \to SO_4,
\quad (x, y) \mapsto (v \mapsto x \cdot v \cdot y),\]
where we regard $x, y \in S^3$ as unit quaternions, $v \in \mathbb{H}$ and $\cdot$ denotes quaternionic multiplication.
If we define $g_{(k_1, k_2)} \colon S^3 \to SO_4,~x \mapsto q(x^{k_1}, x^{k_2})$, then the map
\[ \Z \oplus \Z \to \pi_3(SO_4), \quad (k_1, k_2) \mapsto [g_{(k_1, k_2)}], \]
is an isomorphism, which we use as co-ordinates for $\pi_3(SO_4)$.
For example, by \cite[Ch.\ 8, Prop.\ 12.10]{h94}, the map $g_{(1,1)}$ 
is a clutching function for $TS^4$ and so $\tau_4 = [g_{(1,1)}]$.

\begin{definition} \label{d:Ek1k2}
For $(k_1, k_2) \in \Z^2$, let $E_{k_1, k_2} \to S^4$ be the oriented rank-$4$ vector bundle 
with clutching function $g_{(k_1, k_2)}$; e.g.\ $TS^4 \cong E_{1, 1}$.
\end{definition}

\begin{definition} \label{d:sigma4}
We define $\sigma_4 := [g_{(0, 1)}]$.
\end{definition}

The turning obstructions $\TO_+, \TO_-$ and $\TO_\eta$ take values in the group $\pi_4(SO_4)$ and we
use the isomorphism $q_* \colon \pi_4(S^3 \times S^3) \to \pi_4(SO_4)$
to identify $\pi_4(SO_4) \cong \pi_4(S^3) \oplus \pi_4(S^3) \cong \Z/2 \oplus \Z/2$.

\begin{theorem} \label{t:S4}
The turning obstructions for rank-$4$ bundles over $S^4$ are given follows:
\begin{compactenum}[a)]
\item $\TO_+(a, b) = (\rho_2(a), 0)$;
\item $\TO_-(a, b) = (0, \rho_2(b))$;
\item $\TO_\eta(a, b) = (\rho_2(a), \rho_2(b))$.
\end{compactenum}
\end{theorem}

\begin{proof}
We first prove part c).
Consider the maps $g_{(1, 0)} \colon S^3 \to SO_4$,
$g_{(0, 1)} \colon S^3 \to SO_4$, which are continuous group homomorphisms
and the map $\mathrm{ev} \colon SO_4 \to S^3$.
Since $\mathrm{ev} \circ g_{(1, 0)} = \mathrm{ev} \circ g_{(0, 1)} = \mathrm{Id}_{S^3}$,
it follows by Lemma~\ref{l:L_circ_p} and~\ref{l:p_*L_*} b) that $\TO_\eta(1, 0) \neq 0$ and $\TO_\eta(0, 1) \neq 0$.
Now $\TO_\eta$ is natural for continuous homomorphisms of topological groups
and it follows that
\[ \TO_\eta(1, 0) = (\rho_2(1), 0) 
\quad \text{and} \quad
\TO_\eta(0, 1) = (0, \rho_2(1)).\]

We prove parts a) and  b) together.
Let $i, j, k \in \mathbb{H}$ be the standard purely imaginary unit quaternions.
If we take the standard complex structure on 
$\mathbb{H} = \C \oplus \C j$ to be given by left multiplication by $i \in \mathbb{H}$,
then $g_{(0, 1)}(x)$ commutes with $i$ for every $x \in S^3$ and so $g_{(0, 1)}(x) \in U_2 \subset SO_4$ for
all $x \in S^3$.  Thus $E_{0, 1}$ admits a complex structure and so $\TO_+(0, 1) = 0$.  
Since $ji = -k$, we see that right multiplication by $i$ defines a complex structure on $\mathbb{H}$ 
whose induced orientation
is opposite to the standard orientation.
Since $g_{(1, 0)}(x)$ commutes with right multiplication by $i$ for all $x \in S^3$,
we see that $\bar E_{1, 0}$ admits a complex structure.  
Hence $\TO_-(1, 0) = 0$.
Now we have
\[ \TO_+(a, b) = \TO_+(a, 0) = \TO_-(a, 0) + \TO_\eta(a, 0) = (\rho_2(a), 0)  \]
and
\[ \TO_-(a, b) = \TO_-(0, b) = \TO_+(0, b) + \TO_\eta(0, b) = (0, \rho_2(b)).  \]
\end{proof}

We now discuss the relationship between turning type of rank-$4$ bundles and the
the homotopy classification of their gauge groups, due to
Kishimoto, Membrillo-Solis and Theriault \cite{t21}.
Following the notation of \cite{t21}, 
let $\GG_{k_1, k_2}$ denote the gauge group of $E_{k_1, k_2}$.
Let $\{\!\{a,b\}\!\}$ denote the multiset consisting of the elements $a,b$
and for integers $a$ and $b$ write $(a,b)$ for their greatest common divisor.
Then, for integers $r,a$ and $b$, write $M^r(a,b)$ for the multiset $\{\!\{(a,r),(b,r)\}\!\}$.
By \cite[Theorem 1.1(b)]{t21}
if $\GG_{k_1,k_2} \simeq \GG_{l_1,l_2}$ then $M^{4}(k_1,k_2)=M^4(l_1,l_2)$.

Recall that the turning type of an orientable bundle is characterised by whether it is 
either bi-turnable, strongly chiral or not turnable.  By Theorem~\ref{t:S4},
$E_{k_1, k_1}$ is bi-turnable if $M^2(k_1, k_2) = \{\!\{2, 2\} \!\}$, strongly chiral if
$M^2(k_1, k_2) = \{\!\{1, 2\} \!\}$ and not turnable if $M^2(k_1, k_2) = \{\!\{1, 1\} \! \}$.
Hence combining \cite[Theorem 1.1(b)]{t21} and Theorem~\ref{t:S4} we have

\begin{proposition} \label{p:tt}
The turning type of $E_{k_1, k_2}$ is a homotopy invariant of $\GG_{k_1, k_2}$. \qed
\end{proposition}

\subsection{The turning obstructions $\TO_\pm$} \label{ss:TO_S2k}
In this subsection, we complete the proof of Theorem~\ref{t:to_S2k}.

\begin{proof}[The proof of parts a)-c) of Theorem~\ref{t:to_S2k}]
a) If $k$ is odd, the fact that $\TO_+(\xi) = \TO_{-}(\xi)$ follows from Theorem \ref{thm:gpd-appl} b) and Lemma \ref{l:tospn-kodd}.
If $k = 1, 3$, then $\pi_{2k}(SO_{2k}) = 0$ and the statement holds trivially.
If $k \geq 5$ is odd, then by Corollary~\ref{c:SKirchoff}, $\TO_+(\tau_{2k}) \neq 0$.
Since $k$ is odd, $2 \TO_+(\tau_{2k}) = \TO_\eta(\tau_{2k})= 0$ by Theorem \ref{thm:gpd-appl} d) and Lemma \ref{l:tospn-kodd}.
Since $\pi_{2k}(SO_{2k}) \cong \Z/4$ and $\mathrm{e}(\tau_{2k}) = 2$, the 
result holds for $\Z(\tau_{2k}) \subseteq \pi_{2k-1}(SO_{2k})$.
If $k \equiv 3$ mod~$4$ then $\pi_{2k-1}(SO_{2k}) = \Z(\tau_{2k})$.
If $k \equiv 1$ mod~$4$, then $\pi_{2k-1}(SO_{2k}) = \Z(\tau_{2k}) \oplus \Z/2(\sigma_{2k})$,
provided $k \geq 5$ as we are assuming.  Hence it suffices to show that
$\TO_+(\sigma_{2k}) = 0$.  Now $\pi_{2k-1}(U) \to \pi_{2k-1}(SO)$ is onto
and so $S(\sigma_{2k})$ is stably complex.
By \cite[Ch.\ 20, Corollary 9.8]{h94} $S(\sigma_{2k})$ admits a complex structure with 
$c_k(\sigma_{2k}) = (k{-}1)!$ and since $k \geq 5$, $(k{-}1)!$ is divisible by $4$.
 Hence $\sigma_{2k} - \frac{(k-1)!}{2}\tau_{2k}$ admits a complex structure.
 Then
 
 \[ 0 = \TO_+\!\left( \sigma_{2k} - \frac{(k-1)!}{2} \tau_{2k} \right) = 
 \TO_+(\sigma_{2k}) - \frac{(k-1)!}{4} \bigl( 2 \TO_+(\tau_{2k}) \bigr) 
 = \TO_+(\sigma_{2k}).\]

b) The special case $k = 2$ is proven in Theorem~\ref{t:S4}.

For $k \geq 6$, we first prove that $\mathrm{ev}_*(\TO_\pm(\tau_{2k})) = 1$.  
By Corollary~\ref{c:SKirchoff}, $\TO_\pm(\tau_{2k}) \neq 0$.  
By Theorem \ref{thm:gpd-appl} b) and Theorem~\ref{t:to_S2k} d), 
we have $\TO_{-}(\tau_{2k}) = \TO_+(\tau_{2k}) + \tau_{2k}\eta_{2k-1}$.  
Since $\mathrm{Ker}(\mathrm{ev}_*)$ is generated by $\tau_{2k}\eta_{2k-1}$,
it follows that $\mathrm{ev}_*(\TO_\pm(\tau_{2k})) = 1$.
As $k \geq 6$, $\mathrm{e}(\sigma_{2k}) = 0$ and we must show that $\TO_\pm(\sigma_{2k}) = 0$.
By Lemma~\ref{l:sigma2k}, $\sigma_{2k} \in S(\pi_{2k-1}(SO_{2k-2}))$ and so $\TO_\eta(\sigma_{2k}) = 0$ and
so it suffices to show that $\TO_+(\sigma_{2k}) = 0$.
The argument is analogous for the case $k \equiv 1$ mod~$4$. 

c) We first prove that $\mathrm{ev}_*(\TO_+(\tau_{2k})) = 1$.
Since $\tau_{2k}$ is stably trivial, so is $\TO_+(\tau_{2k})$.
The proof is now the same as the proof when $k \equiv 2$~mod~$4$.

To see that $S(\TO_+(\sigma_{2k})) = 1$, we note that $S(\sigma_{2k})$ generates
$\pi_{2k-1}(SO) \cong \Z$ and that the natural map $\pi_{2k-1}(U) \to \pi_{2k-1}(SO)$ has image the subgroup of 
of index $2$ by \cite{b59}.  Hence $\sigma_{2k}$ does not admit a stable complex structure and
so $\sigma_{2k}$ is not stably turnable by Theorem~\ref{t:tc}.  
Now by Lemma~\ref{l:to+stab}, $S(\TO_+(\sigma_{2k})) = \TO_+(S(\sigma_{2k})) = 1$.
\end{proof}

\section{Stable turnings and stable complex structures}  \label{s:stable}
In this section we define stable turnings of vector bundles $E \to B$.
When $B$ has the homotopy type of a finite $CW$-complex, 
we will see that Bott's proof of Bott periodicity shows that the space of stable complex
structures on $E$ is weakly homotopy equivalent to the space of stable turnings of $E$.
In particular, in this case $E$ is stably turnable if and only if $E$ admits a stable complex structure.

Recall that $\ul \R^j$ denotes the trivial vector bundle over $B$ of rank $j$.

\begin{definition}[Stably turnable]
A rank-$n$ vector bundle $E \to B$ is {\em stably turnable}
if $E \oplus \ul \R^j$ is turnable for some $j \geq 0$.
\end{definition}

\noindent
Of course, if $E$ is turnable, then $E$ is stably turnable.

\begin{remark}
In the definition of stably turnable, $n{+}j$ must be even but $n$ need not be even.
\end{remark}

\begin{remark}
\label{rem:stable}
It is clear from definition of turning that if $E \oplus \ul \R^j$ is turnable, 
then $E \oplus \ul \R^{j+2l}$ is turnable for any non-negative integer $l$.
\end{remark}
 
For any rank-$n$ vector bundle $E \to B$, recall that $\Fr(E)$, the frame bundle of $E$,
is the principal $SO_n$-bundle associated to $E$
and for any non-negative integer $j$ with $n{+}j$ even, 
\[ \Turn(E \oplus \ul \R^j) = \Fr(E \oplus \ul \R^j) \times_{SO_{n+j}}\Omega_{\pm \idbb}SO_{n+j}\]
is the associated turning bundle of $E \oplus \ul \R^j$.
Now orthogonal sum with the path $\beta \in \Omega_{\pm \idbb}SO_2$
defines the injective map
$i_\beta \colon \Omega_{\pm \idbb}SO_{n+j} \to \Omega_{\pm \idbb}SO_{n+j+2}$,
which we regard as an inclusion.
Thus we regard
 $\Turn(E \oplus \ul \R^j)$ as a subbundle of 
$\Turn(E \oplus \ul \R^{j+2})$ and
we set
\[ \Turn(E^\infty) := \bigcup_{n+j ~ \text{even}} \Turn(E \oplus \ul \R^j), \]
which is a fibre bundle over $B$ with fibre $\Omega_{\pm \idbb}SO := \bigcup_{j=1}^{\infty} \Omega_{\pm \idbb}SO_{2j}$.

\begin{lemma} \label{l:omega}
Let $E \to B$ be a vector bundle over a space homotopy equivalent to a finite $CW$-complex. 
Then $E$ is stably turnable if and only if $\Turn(E^\infty) \to B$ admits a section.
\end{lemma}

\begin{proof}
Remark \ref{rem:stable} and Lemma \ref{l:TurnE} tell us that a vector bundle $E \to B$ is stably turnable  implies that $\Turn(E^\infty) \to B$ admits a section.
Conversely, note that the inclusion map $\Omega_{\pm \idbb}SO_{2j} \to \Omega_{\pm \idbb}SO$ is $(2j{-}2)$-connected and $B$ is homotopy equivalent to a finite $CW$-complex, it follows from the obstruction theory (cf.' \cite[Chapter VI, Section 5]{w78}) that if $\Turn(E^\infty) \to B$ admits a section, then there must exists a non-negative integer $j$ such that $\Turn(E \oplus \ul \R^j)$ admits a section, which means that $E \to B$ is stably turnable and the proof is complete.
\end{proof}

\noindent
Given Lemma~\ref{l:omega}, an efficient way to define the notion of a stable turning is via a section of $\Turn(E^\infty)$.

\begin{definition}[Stable turning and the space of stable turnings]
A {\em stable turning} of a vector bundle $E \to B$ is a section of the fibre bundle $\Turn(E^\infty) \to B$.
The {\em space of stable turnings} of $E$, $\Gamma(\Turn(E^\infty))$, is the space of sections of 
$\Turn(E^\infty) \to B$, equipped with the restriction of compact-open topology.
\end{definition}

We next consider {\em minimal turnings}, which are turnings that restrict to minimal geodesics in each fibre.
The manifold $SO_{2k}$ has a canonical Lie invariant metric, which allows us to 
consider geodesics in $SO_{2k}$.
We write $\Omegamin \subset \Omegapmone$ 
for the subspace of paths which are minimal geodesics in $SO_{2k}$ from $\idbb$ to $-\idbb$. 
Since the conjugation action of $SO_{2k}$ on itself is by isometries, it preserves geodesics and so
$\Omegamin$ is an $SO_{2k}$-subspace of $\Omegapmone$.
For any non-negative integer $j$ with $n{+}j$ even, 
denote by $\Turnmin(E \oplus \ul \R^j) \subset \Turn(E \oplus \ul \R^j)$
the subbundle of minimal geodesic turnings and set
\[ \Turnmin(E^\infty) 
:= \bigcup_{n+j ~ \text{even}} \Turnmin(E \oplus \ul \R^j), \]
which is a fibre bundle over $B$ with fibre 
$\Omega_{\pm \idbb}^{\mathrm{min}}SO = \cup_{j=1}^{\infty} \Omega^{\mathrm{min}}_{\pm \idbb} SO_{2j}$.

Now
$\Omega_{\pm \idbb}^{\mathrm{min}}SO$ is an $SO$-subset of $\Omega_{\pm \idbb}SO$ and hence 
$\Turnmin(E^\infty)$ is a subbundle of $\Turn(E^\infty)$.
Since the inclusion $\Omega_{\pm \idbb}^{\mathrm{min}}SO \to \Omega_{\pm \idbb}SO$ is a weak homotopy equivalence \cite[Theorem 24.5]{m63} and
Part(a) of the next lemma then follows immediately.  Part(b) follows from Lemma \ref{l:omega}.

\begin{lemma} \label{l:Om}
Let $E \to B$ be a vector bundle over a space homotopy equivalent to a finite $CW$-complex. 
Then the following hold:
\begin{compactenum}[a)]
\item The fibrewise inclusion $\Turnmin(E^\infty) \to \Turn(E^\infty)$ is a weak homotopy equivalence;
\item $E$ is stably turnable if and only if $\Turnmin(E^\infty) \to B$ admits a section.  \qed
\end{compactenum}

\end{lemma}

Now we recall the relationship of complex structures on a bundle and minimal turnings.
A complex structure on a rank-$2k$-vector bundle $E$ is an element $J \in \GG_E$
such that $J^2 = -\idbb$.  In particular, a complex structure on $E$ endows each
fibre of $E$ with the structure of a complex vector space.
If $E$ has rank $n$, then a stable complex structure on $E$ is a complex structure on $E \oplus \ul \R^j$ for some $j \geq 0$ with $n{+}j$ even.

Now let 
\[ \mathcal{J}_{2k} := \{ J \in SO_{2k} ~ | ~ J^2 = - \idbb\} \]
be the space of (special orthogonal) complex structures on $\R^{2k}$.
The space $\mathcal{J}_{2k}$ is as an $SO_{2k}$-space, 
where $SO_{2k}$ acts on $\mathcal{J}_{2k}$ by conjugation.
For any non-negative integer $j$ with $n{+}j$ even, define $\mathcal{J}(E \oplus \ul \R^j)  \subset \Aut(E \oplus \ul \R^j)$
\[ \mathcal{J}(E \oplus \ul \R^j) := \Fr(E \oplus \ul \R^j) \times_{SO_{n+j}} \mathcal{J}_{n+j} \]
to be the bundle of fibrewise complex structures on $E \oplus \ul \R^j$.
Regarding $\GG_E = \Gamma(\Aut(E))$ we see that a complex structure on $E$
is equivalent to a section of $\mathcal{J}(E) \to B$ and that a stable complex 
structure on $E$ is equivalent to a section of $\mathcal{J}(E \oplus \ul \R^j) \to B$.
Letting $j$ tend to infinity, 
we define
\[ \mathcal{J}(E^\infty) := \bigcup_{n+j ~ \text{even}} \mathcal{J}(E \oplus \ul \R^{n+j}), \]
to be the bundle of fibrewise stable complex 
structures on $E$.  The space $\mathcal{J}(E^\infty)$ is the total space 
of a bundle over $B$ with fibre 
$\mathcal{J}_\infty := \cup_{j=1}^{\infty}\mathcal{J}_{2j}$
and we have

\begin{lemma} \label{l:J}
A vector bundle $E \to B$ over a space homotopy equivalent to a finite $CW$-complex
admits a stable complex structure if and only if $\mathcal{J}(E^\infty) \to B$ admits a section. \qed
\end{lemma}

\noindent
In the light Lemma~\ref{l:J}, an efficient way to define the notion of a stable 
complex structure is via a section of $\mathcal{G}(E^\infty)$.

\begin{definition}[Stable complex structure and the space of stable complex structures]
A {\em stable complex structure} on an oriented vector bundle $E \to B$ is a section of the fibre bundle 
$\mathcal{J}(E^\infty) \to B$.
The {\em space of stable complex structures} of $E$, $\Gamma(\mathcal{J}(E^\infty))$, is the space of sections of 
$\mathcal{J}(E^\infty) \to B$, equipped with the restriction of compact-open topology.
\end{definition}

Now a complex structure on $\R^{2k}$ defines a minimal geodesic in $SO_{2k}$ via complex multiplication with unit complex numbers in the upper half plane.  Explicitly, we define the map
\[ \varphi_{2k} \colon \mathcal{J}_{2k} \to \Omegamin, \quad
J \mapsto (t \mapsto \exp(\pi t J)),\] 
where 
$\exp \colon T_\idbb SO_{2k} \to SO_{2k}$ 
is the exponential map from the tangent space over the identity.
Then $\varphi_{2k}$ is an $SO_{2k}$-equivariant homeomorphism (see, e.g.\ Milnor \cite[Lemma 24.1]{m63}).
It follows for any rank-$2k$-bundle $E \to B$ that $\varphi_{2k}$ induces a fibrewise homeomorphism
$\varphi_{2k} \colon \mathcal{J}(E) \to \Turnmin(E)$.  Stably, we define the map
$\varphi_\infty := \lim_{j \to \infty} \varphi_{2j} \colon \mathcal{J}_\infty \to \Omega_{\pm \idbb}^{\mathrm{min}}SO$
and we have

\begin{lemma} \label{l:J_to_Omin}
Let $E \to B$ be a vector bundle.
The map $\varphi_\infty  \colon \mathcal{J}_\infty \to \Omega^{\mathrm{min}}_{\pm \idbb}SO$
 induces a fibrewise homeomorphism
$\varphi_\infty \colon \mathcal{J}(E^\infty) \to \Turnmin(E^\infty))$.
\qed
\end{lemma}

\noindent
Now the fibre bundle map $\mathcal{J}(E^\infty) \to \Turnmin(E^\infty) \to \Turnmin(E)$ induces a map
$\Gamma(\mathcal{J}(E^\infty)) \to \Gamma(\Turn(E^\infty))$.
By combining Lemmas~\ref{l:Om}(a), \ref{l:J} and \ref{l:J_to_Omin} we obtain

\begin{theorem} \label{t:tc}
Let $E \to B$ be a vector bundle over a space homotopy equivalent to a finite $CW$-complex.
The induced map $\Gamma(\mathcal{J}(E^\infty)) \to \Gamma(\Turn(E^\infty))$ from the space of stable complex structures on $E$ to the space of stable turnings on $E$ weak homotopy
equivalence.
Hence $E$ is stably turnable if and only if $E$ admits a stable complex structure. 
\qed
\end{theorem}

In the remainder of this section we present an alternative proof of the final sentence of Theorem \ref{t:tc} 
using $K$-theory.
Let $BU$ (resp.\ $BO$) be the classifying space of the stable unitary group $U$ (resp.\ stable orthogonal group $O$).
Since $O/U$ is homotopy equivalent to $\Omega^2 BO$ (cf.\ \cite{b59}), the canonical fibration
$$O/U\hookrightarrow BU\rightarrow BO$$
gives rise to the Bott exact sequence (cf.\ Bott \cite[(12.2)]{bk69} or Atiyah \cite[(3.4)]{a66}):
\begin{equation} \label{eq:bs}
  \cdots \rightarrow KO^{-2}(B) \rightarrow \widetilde{KU}(B)\xrightarrow{r} \widetilde{KO}(B)\xrightarrow{\del} \wt KO^{-1}(B)\rightarrow\cdots
\end{equation}
Here $r$ is the real reduction homomorphism 
and $\del$ is the homomorphism given by $\del(\xi) = \eta \cdot \xi$
where $\eta$ is the generator of $KO^{-1}(\pt) = \Z/2$ and $\cdot$ denotes the product
in real $K$-theory.

Now for a rank-$n$ vector bundle $E \to B$, let $\xi_E \in \wt{KO}(B)$ be the
real $K$-theory class $\xi_E := E \ominus \ul \R^n$, which is represented by the virtual bundle obtained
as the formal difference of $E$ and the trivial rank-$n$ bundle over $B$.
When $B$ is homotopy equivalent to a finite $CW$-complex, 
the bundle $E$ admits a stable complex structure if and only if the real $K$-theory 
class $\xi_E$ lies in the image of the real reduction homomorphism $r$.  
Hence the next proposition follows from the Bott exact sequence above.

\begin{proposition} \label{p:eta}
Let $E \to B$ be a vector bundle over a space homotopy equivalent to a finite $CW$-complex.
Then $E$ admits a stable complex structure if and only if $\eta \cdot \xi_E= 0$. \qed
\end{proposition}

We now relate the boundary map $\del$ in \eqref{eq:bs} to turnings of vector bundles.  Let $\psi \colon E \to E$ be an
automorphism of a vector bundle $\pi \colon E \to B$. 
The {\em mapping torus} of $\psi$ is the vector bundle $T(\psi) \to B \times S^1$,
where
\[ T(\psi) := (E \times I)/\simeq \]
with $(v, 0) \simeq (\psi(v), 1)$ and the bundle map is given by $[(v, t)] \mapsto (\pi(v), [t])$.
We note that $T(\psi)$ is orientable if and only if $\psi$ is orientation-preserving,
in which case $T(\psi)$ inherits an orientation from $E$.
Let $\wh \otimes$ denote the exterior tensor product of vector bundles.

\begin{lemma} \label{l:Tpsi}
Let $\ul \R$ denote the trivial line bundle over $S^1$.
The bundle $T(\psi)$ is isomorphic to $E \etp \ul{\R}$ if and only if $\psi$ 
is homotopic to the identity.
\end{lemma}

\begin{proof}
The classification of vector bundles \cite[Chapter 3, Section 4]{h94} shows that a vector bundle isomorphism $E \to E'$
is equivalent to a vector bundle $F \to B \times I$ such that $F|_{B \times \{0\}} = E$
and $F|_{B \times \{1\}} = E'$.  Similarly, a vector bundle automorphism $E \to E$
is equivalent to a vector bundle $F \to B \times S^1$ such that $F|_{B \times \{1\}} = E$.
In particular, the bundle $\ul \R \etp E$ corresponds to the identity automorphism $\idbb_E \colon E \to E$.
It follows that a vector bundle isomorphism $T(\psi) \to \ul \R \times E$ is equivalent to
a bundle over $B \times S^1 \times I$ and so is equivalent to a path of bundle automorphisms
from $\psi$ to $\idbb_E$.
\end{proof}

\begin{proof}[An alternative proof of  the final sentence of Theorem \ref{t:tc}]
Let $L \to S^1$ be the M\"obius bundle; i.e.\,the non-trivial rank-$1$ bundle.
Then $\eta \in KO^{-1}(\pt) = \wt{KO}(S^1)$ is represented by the virtual bundle
$L \ominus \ul \R$.
Let $\xi_E$ be represented by the virtual bundle $E' \ominus \ul \R^{2k}$,
where $E' \to B$ is a vector bundle of rank-$2k$, which is larger than the formal dimension of $B$.
Then $\eta \cdot \xi_E \in KO^{-1}(B) = KO(B \wedge S^1)$ is represented by the virtual
bundle 
\[ (E' \etp L) \oplus (\ul \R^{2k} \etp \ul \R) \ominus  (\ul \R^{2k} \etp L) \ominus (E' \etp \ul \R)\]
over $B \times S^1$, 
which is canonically zero over $B \vee S^1$.
Since $L$ is the mapping torus of $-\idbb_{\ul \R} \colon \ul \R \to \ul \R$, 
it follows that $\ul \R^{2k} \etp L$ is the mapping torus of $-\idbb_{\ul \R^{2k}} \colon \ul \R^{2k} \to \ul \R^{2k}$
and so is trivial.
Hence $\eta \cdot \xi_E = 0$ if and only if $E' \etp L$ is stably isomorphic
to $E' \etp \ul \R$.

Since $L$ is the mapping torus of $-\idbb_{\ul \R} \colon \ul \R \to  \ul \R$, 
the bundle $E' \etp L$ is the mapping tours of $-\idbb_{E' }\colon E' \to E'$.
Now by Lemma~\ref{l:Tpsi}, $E' \etp L$
is isomorphic to $E' \etp \ul \R$ if and only if $E'$ is turnable.
Hence $\eta_E \cdot \xi = 0$ if and only if $E'$ is stably turnable and so the final sentence of
Theorem \ref{t:tc} follows directly from Proposition \ref{p:eta}.
\end{proof}

\section{Turning rank-$2k$-bundles over $2k$-complexes} \label{s:2k_over_2k}
In Section~\ref{s:stable} we saw that a turning of bundle $E$ induces a stable complex
structure on $E$ and in Section~\ref{s:2k_over_S2k} we computed the turning obstruction obstruction
for rank-$2k$-bundles over the $2k$-sphere.  In this section we combine these results
to gain useful information about the turning obstruction
for rank-$2k$ bundles over $2k$-dimensional complexes.
Throughout this section $B$ will be a space, which is homotopy equivalent to
a connected finite $CW$-complex of dimension $2k$ or less.

Theorem~\ref{t:TC1} below gives a necessary condition for an oriented rank-$2k$ bundle $E$
over $B$ to be positive-turnable.
Its statement requires some preliminary definitions.
We will say that a
complex structure $J$ on $E \oplus \ul \R^{2j}$ for some $j \geq 0$
is {\em compatible with $E$}, if
$J$ induces the same orientation on $E \oplus \ul \R^{2j}$ as $E$ does.
Recall that $\times 2 \colon H^*(B; \Z/2) \to H^*(B; \Z/4)$ 
is the map induced by the inclusion of coefficients $\times 2 \colon \Z/2 \to \Z/4$ 
and the subgroup $I^{2k}(X) \subseteq H^{2k}(B; \Z/4)$, which is defined by
\[ I^{2k}(B) = 
\begin{cases}
((\times 2) \circ Sq^2 \circ \rho_2)(H^{2k-2}(B; \Z)) & \text{$k$ is odd}, \\
0 & \text{$k$ is even}.
\end{cases}
\]
Recall also that $\rho_4$ denotes reduction mod~$4$.

\begin{theorem} \label{t:TC1}
Let $E \to B$ be an oriented rank-$2k$ vector bundle.
If $E$ is positive-turnable then for some $j \geq 0$, $E \oplus \ul \R^{2j}$ admits a complex structure 
$J$ such that $J$ is compatible with $E$ and $c_k(J)$ satisfies
\[ [\rho_4(c_k(J))] = [\rho_4(\mathrm{e}(E))] \in H^{2k}(B; \Z/4)/I^{2k}(B) .\] 
\end{theorem}

Example \ref{e:not_suff} below shows that there are turnable bundles which do not satisfy the
condition of Theorem~\ref{t:TC1}.  However, this condition is sufficient for the bundle to
be turnable if $k$ is odd and in many cases if $k$ is even.

\begin{theorem} \label{t:TC2}
Let $E$ be an oriented rank-$2k$ vector bundle over $B$
with either $k$ odd or $k$ even and $H^{2k}(B; \Z)$ $2$-torsion free.
If $E$ admits a stable complex structure $J$ such that 
\[ [\rho_4(c_k(J))] = [\rho_4(\mathrm{e}(E))] \in H^{2k}(B; \Z/4)/I^{2k}(B), \]
then $E$ is positive-turnable. 
\end{theorem}

\begin{example} \label{e:not_suff}
Let $M(\Z/2, 4k{-}1) := S^{4k{-}1} \cup_2 D^{4k}$ be the mod~$2$ Moore space with $H^{4k}(X; \Z) = \Z/2$
and let $c \colon M(\Z/2, 4k{-}1) \to S^{4k}$ by the map collapsing the $(4k{-}1)$-cell.
Since $c$ is the suspension of the map $c' \colon M(\Z/2, 4k{-}2) \to S^{4k-1}$, which collapses the
$(4k{-}2)$-cell of $M(\Z/2, 4k{-}2)$, we see that the $\gamma$-turning obstruction of 
$E := c^*(TS^{4k})$ is the pull-back $c'^*(\TO_\gamma(\tau_{4k})) \in [M(\Z/2, 4k{-}2), SO_{2k}]$.
Since $\tau_{4k} \notin 2 \pi_{4k-1}(SO_{4k})$, it follows that $c'^*(\TO_\gamma(\tau_{4k})) \neq 0$.
By Proposition~\ref{p:to}, $E$ is not $\gamma$-turnable for any path $\gamma$ and so $E$ is not turnable.
However, $E$ is stably parallelisable and so admits a stable complex structure $J$ with $c_{2k}(J) = 0$.
Moreover $\mathrm{e}(E) = 0$, since $\mathrm{e}(TS^{4k}) \in 2 H^{4k}(S^{4k}; \Z)$. 
Hence $E$ satisfies the condition of Theorem~\ref{t:TC1}.
\end{example}

Before proving Theorems \ref{t:TC1}~and~\ref{t:TC2}, we give an application of Theorem~\ref{t:TC2}
which shows, in particular, that for all $l \geq 1$ there are $8l$-manifolds $M$ whose tangent bundles are turnable
but not complex; e.g.~$M = S^4 \times S^4$.

\begin{corollary} \label{c:TC2}
For $i > 0$, let $M$ be an orientable $4i$-manifold such that the following hold:
\begin{compactenum}[a)]
\item $M$ is stably parallelisable;
\item $\chi(M) \neq 0$;
\item $KU(M) \to KO(M)$ is injective.
\end{compactenum}
Then $TM$ does not admit a complex structure but $TM$ is turnable iff $\chi(M) \equiv 0$~mod~$4$.
In particular  for all $m \geq 1$ and $l \geq 0$ the manifolds $M_l := \sharp_l(S^{4m} \times S^{4m})$ are such that 
$TM_l$ does not admit a complex structure but $TM_l$ is turnable 
if and only if $l$ is odd.
\end{corollary} 

\begin{proof}
An simple computation shows that $TM$ admits two homotopy classes of stable complex structure $J$.
Since $TM$ is stably trivial and $KU(M) \to KO(M)$ is injective, $E_J$ the complex bundle underlying $J$,
is trivial, and so and $c_{4j}(J) = 0$. 
On the other hand,
$\mathrm{e}(TM) = \chi(M) = 2 + 2l$ by the Poincar\'e-Hopf Theorem \cite[p.\ 113]{h89}.
Hence by Theorem~\ref{t:stabC_to_C}, 
$TM$ does not admit a complex structure.
However, by Theorems~\ref{t:TC1} and~\ref{t:TC2}, $TM$ is turnable if and only if $l$ is odd.
\end{proof}

We now turn to the proofs of Theorems~\ref{t:TC1} and~\ref{t:TC2}.
Without loss of generality, we may assume that $B$ is a finite $CW$-complex of dimension
at most $2k$.
It will be useful to modify a rank-$2k$-bundle $E \to B$ over the $2k$-cells of $B$ and
we first describe this process.
Let $F \to S^{2k}$ be an oriented rank-$2k$ bundle with clutching function $g$.
Given a $2k$-cell $e^{2k}_\alpha \subset B$ and a $2k$-disc $D^{2k}_\alpha$ 
embedded in the interior of $e^{2k}_\alpha$, we define the bundle $E \sharp_\alpha F \to B$ as follows:
Write $B = B^\circ \cup_{S^{2k-1}_\alpha} D^{2k}_\alpha$, where $B^\circ := B \setminus \mathrm{Int}(D^{2k}_\alpha)$ 
and 
\[ E = E^\circ \cup_\phi (\R^{2k} \times D^{2k}_\alpha),\]
where $E^\circ$ is the restriction of $E$ to $B^\circ$ and 
$\phi \colon E|_{S^{2k-1}_\alpha} \to \R^{2k} \times S^{2k-1}_\alpha$ is a bundle isomorphism and define
\[ E \sharp_\alpha F := E^\circ \cup_{g \circ \phi} (\R^{2k} \times D^{2k}_\alpha).\]
We let $e^{2k*}_\alpha \in C^{2k}(B)$ be the $2k$-dimensional cellular cochain of $B$,
which evaluates to $1$ on the $2k$-cell $e^{2k}_\alpha$ and $0$ on every other $2k$-cell.
The next results follows from an elementary application of obstruction theory.

\begin{lemma} \label{l:E+F}
Let $E_1$ and $E_2$ be rank-$2k$ oriented bundles over $B$.
Then $E_1$ is stably equivalent to $E_2$ if and only if there are $2k$-cells $\alpha_1, \dots, \alpha_n$
of $B$ and integers $j_1, \dots, j_n$ such that
$E_2 \cong E_1 \#_{\alpha_1} j_1 TS^{2k} \#_{\alpha_2} \dots \#_{\alpha_n} j_n TS^{2k}$.
Moreover, in this case the Euler classes of $E_1$ and $E_2$ are represented by cocycles
$\cc(\mathrm{e}(E_1))$ and $\cc(\mathrm{e}(E_2))$ such that
$\cc(\mathrm{e}(E_2)) = \cc(\mathrm{e}(E_1)) + \sum_{i=1}^n 2j_i e^{2k*}_{\alpha_i}$. \qed
\end{lemma}

\noindent
Lemma~\ref{l:E+F} can be used to prove the following theorem of Thomas.

\begin{theorem}[{\cite[Theorem 1.7]{t67}}]
\label{t:stabC_to_C}
Let $E$ be a rank-$2k$ vector bundle over $B$, where $H^{2k}(B; \Z)$ is $2$-torsion free.
Then $E$ admits a complex structure, if and only if $E$ admits a stable complex structure $J$
such that $c_k(J) = \mathrm{e}(E)$. \qed
\end{theorem}

\begin{proof}[Proof of Theorem~\ref{t:TC1}]
Recall the universal rank-$2k$ turning bundle $BT_{2k} = ESO_{2k} \times_{SO_{2k}} \Omegapmone \to BSO_{2k}$ from Definition~\ref{d:univ-turn}
and define $BT := ESO \times_{SO} \Omega_{\pm \idbb} SO \to BSO$ to be the universal stable turning bundle.
There is a natural map $BT_{2k} \to BT$ and 
by the results of Section~\ref{s:stable}, 
$BT \simeq BU$.  Hence we consider the commutative diagram,
\[
\xymatrix{
& 
BU_k \ar[d] \ar[r] &
BU \ar[d]^\simeq \\
& 
BT_{2k} \ar[d] \ar[r]^S &
BT \ar[d] \\
B \ar[r]^(0.425)f \ar[ur]^{\bar f} &
BSO_{2k} \ar[r] &
BSO,
}\]
where $f$ classifies $E$ and $\bar f$ classifies a positive-turning on $E$.
Since the natural map $BU \to BT$ is a fibre homotopy equivalence over $BSO$,
the turning on $E$ induces a stable complex structure $J$ on $E$.
Consider the lifting problem:
\[
\xymatrix{
& 
BU_k \ar[d]\\
B \ar[r]^{S \circ \bar f} \ar@{.>}[ur] & 
BT
} \]
Since the homotopy fibre of $BU_k \to BT$ is $2k$-connected, the map 
$J : = S \circ \bar f \colon Y \to BT$
has a lift to $J' \colon Y \to BU_k$, which is unique up to homotopy.
The map $J'$ defines a complex rank-$k$ bundle $(E', J')$ over $X$, 
where $E'$ is stably equivalent to $E$.

By Lemma~\ref{l:E+F},
there is a bundle isomorphism 
\begin{equation} \label{eq:EtoE'}
\alpha \colon E \to E' \#_{\alpha_1} j_1 TS^{2k} \# \dots \#_{\alpha_n} j_n TS^{2k}.
\end{equation}
We set $B^\circ := B \setminus \left( \bigcup_{i=1}^n \mathrm{Int}(D^{2k}_{\alpha_i}) \right)$,
$E^\circ := E|_{B^\circ}$,
$E'^\circ := E'|_{B^\circ}$ and $\alpha^\circ := \alpha|_{E^\circ}$.
The lift $\bar f$ defines a turning $\psi_t$ on $E$, which restricts to a turning 
$\psi_t^\circ$ on $E^\circ$ and the complex structure $J'$ defines a turning $\psi'_t$ on
$E'$, which restricts to a turning $\psi_t'^\circ$ on $E'^\circ$ that pulls back along $\alpha^\circ$
to a turning $(\alpha^\circ)^*(\psi_t'^\circ)$ on $E^\circ$.

If $\psi_t^\circ$ and $(\alpha^\circ)^*(\psi_t'^\circ)$ are equivalent turnings on $E^\circ$,
then the obstruction to extending $(\alpha^\circ)^*(\psi_t'^\circ)$ to $E$ vanishes.
On the other hand, the obstruction to extending $\psi_t'^\circ$ to $E'$ vanishes.
It follows that the cocycle $\sum_{i=1}^n j_i \TO_+(\tau_{2k})e^{2k*}_{\alpha_i}$
represents $0$ in $H^{2k}(X; \pi_{2k}(SO_{2k}))$ and so
the cocycle $\sum_{i=1}^n j_i 2e^{2k*}_{\alpha_i}$ represents $0$ in $H^{2k}(B; \Z/4)$
and hence $\rho_4(c_k(J)) = \rho_4(\mathrm{e}(E))$.

If $\psi_t^\circ$ and $(\alpha^\circ)^*(\psi_t'^\circ)$ are not equivalent turnings on $E^\circ$,
then they differ over the $(2k{-}2)$-cells and $(2k{-}1)$-cells of $B^\circ$.
If $k$ is even, this variation does not effect the cohomology class represented by
$\sum_{i=1}^n j_i  \mathrm{ev}_*(\TO_+(\tau_{2k})e^{2k*}_{\alpha_i})$.
If $k$ is odd, then changing the turning can alter the values of 
the turning obstruction over the cells $e^{2k}_{\alpha_i}$ by any element of
$((\times 2) \circ Sq^2 \circ \rho_2)(H^{2k-2}(B; \Z))$.
\end{proof}

\begin{proof}[Proof of Theorem~\ref{t:TC2}]
Suppose that $E$ admits a stable complex structure $J$.
Then, as in the proof of Theorem~\ref{t:TC1}, $J$ descends to a complex rank-$k$ bundle
$(E', J')$ on $X$ such that $E$ and $E'$ are stably isomorphic.
Since $(E', J')$ is a complex bundle, $E'$ is positive-turnable.  
Consider the bundle isomorphism $\alpha$ from \eqref{eq:EtoE'}.
Assume first $k$ is even and that $H^{2k}(X; \Z)$ contains no $2$-torsion.
Since $J'$ stabilises to $J$,
$c_k(J') = c_k(J)$ and so $\rho_4(c_k(J')) = \rho_4(c_k(J)) =  \rho_4(\mathrm{e}(E))$.
It follows that each $j_i$ in \eqref{eq:EtoE'} is even.
Now $2 TS^{2k}$ is positive-turnable by Theorem~\ref{t:to_S2k}.
Since $E'$ is positive-turnable, 
it follows that $E$ is positive-turnable.

When $k$ is odd the argument is similar to the case where $k$ is 
even but requires adjustments.  We first note that
$\rho_2(c_k(J')) = w_{2k}(E) = \rho_2(\mathrm{e}(E))$.
Moreover, for $k$ odd $\TO_+(\tau_{2k}) = \tau_{2k} \eta_{2k-1} \in 2 \pi_{2k}(SO_{2k})$.
It follows that if some $j_i$ in \eqref{eq:EtoE'} is odd, then
we can modify the turning of $E'$ to ensure that the pull-back of its restriction
to $E'^\circ$ extends to all of $E$.
\end{proof}

\begin{proof}[Proof of Theorem~\ref{t:4} and Remark~\ref{r:1}]
These statements follow from Theorems~\ref{t:TC1} and~\ref{t:TC2} and the fact that a turnable bundle
over a connected space is either positively turnable or negatively turnable.
\end{proof}

\begin{question}
From the proof of Theorem~\ref{t:TC1} we see that there is a map $BT_{2k} \to BU$,
obtained by composing the stabilisation map $BT_{2k} \to BT$ with a homotopy inverse
of the natural map $BU \to BT$.
This map $BT_{2k} \to BU$ encodes the fact that a turned bundle $(E, \psi_t)$ over a finite 
$CW$-complex defines a stable
turning of $E$ and a stable complex structure $J$ on $E \oplus \ul \R^j$ for some $j$.
It follows that the universal bundle $VT_{2k} \to BT_{2k}$ has a well-defined homotopy class
of stable complex structure and hence there are well-defined Chern classes
$c_i(VT_{2k}) \in H^{2i}(BT_{2k}; \Z)$.

We pose the following question: What special properties, if any, do the Chern classes of $VT_{2k}$ possess?
For example, one would expect that $c_i(VT_{2k}) = 0$ for all $i > k$
and presumably also that there is an equality
 $[\rho_4(c_k(VT_{2k}))] = [\rho_4(\mathrm{e}(VT_{2k}))] \in H^{2k}(BT_{2k}; \Z/4)/I^{2k}(BT_{2k})$.
\end{question}

\begin{multicols}{2}[]

\bigskip
\noindent
\emph{Diarmuid Crowley}\\
  \vskip -0.125in \noindent
  {\small
  \begin{tabular}{l}
    School of Mathematics and Statistics\\
    University of Melbourne\\
    Parkville, VIC, 3010, Australia\\
    \textsf{dcrowley@unimelb.edu.au}
	
  \end{tabular}}

\bigskip
\noindent
\emph{Blake Sims}\\
  \vskip -0.125in \noindent
  {\small
  \begin{tabular}{l}
  School of Mathematics and Statistics\\
    University of Melbourne\\
    Parkville, VIC, 3010, Australia\\
    \textsf{blake.sims@unimelb.edu.au}
	
  \end{tabular}} 

\bigskip
\noindent
\emph{Csaba Nagy}\\  
\vskip -0.125in \noindent
{\small
  \begin{tabular}{l}
   School of Mathematics and Statistics\\
   University of Glasgow\\
   University Place, Glasgow G12 8QQ, United Kingdom\\
   \textsf{csaba.nagy@glasgow.ac.uk}
   
  \end{tabular}}

\bigskip
\noindent
\emph{Huijun Yang}\\  
\vskip -0.125in \noindent
{\small
  \begin{tabular}{l}
   School of Mathematics and Statistics\\
   Henan University\\
   Kaifeng, Henan, 475004, China\\
   
    \textsf{yhj@amss.ac.cn}
    
  \end{tabular}}
\end{multicols}

\begin{thebibliography}{999}
\bibitem{a66}
M. F. Atiyah, {\em $K$-theory and Reality}, Quart.\ J.\ Math.\ Oxford (2), {\bf 17} (1966), 367--386.

\bibitem{bm64}
M. G. Barratt and M. E. Mahowald, {\em The metastable homotopy of $O(n)$}, 
Bull.\ Amer.\ Math.\ Soc.\ {\bf 70} (1964), 758--760.

\bibitem{b-s51} A. Borel and and J. P. Serre,
 {\em D\'{e}termination des {$p$}-puissances r\'{e}duites de {S}teenrod 
 dans la cohomologie des groupes classiques. Applications},
C.\ R.\ Acad.\ Sci.\ Paris, {\bf 233} (1951), 680--682.

\bibitem{b59} 
R. Bott, {\em The stable homotopy of the classical groups}, 
Ann.\ of Math.\ {\bf 70} (1959), 313--337.

\bibitem{bk69}
R. Bott, {\em Lectures on $K(X)$}, Mathematics Lectures Note Series, W.A. Benjamin, Inc., 
New York-Amsterdam, 1969.

\bibitem{bm58} R. Bott and J. Milnor,
{\em On the parallelizability of the spheres},
Bull.\ Amer.\ Math.\ Soc.\ {\bf 64} (1958), 87--89.

\bibitem{br72} W.~Browder, 
{\em Surgery on simply-connected manifolds}, 
Ergebnisse der Mathematik und ihrer Grenzgebiete, Band 65. Springer-Verlag, New York-Heidelberg 1972. 

\bibitem{hk07} H. Hamanaka and A. Kono,
{\em A note on the Samelson products in $\pi_*(SO(2n))$ and the group $[SO(2n),SO(2n)]$},
Topology Appl.\ {\bf 154} (2007), 567--572.

\bibitem{h89} H. Hopf,
{\em Differential geometry in the large},
in: Lecture Notes in Mathematics, vol. 1000, Springer, 1989.

\bibitem{h94} D.~Husemoller, {\em Fibre bundles}, 
Third edition, GTM 20, Springer-Verlag, New York 1994.

\bibitem{k59}
M.~A.~Kervaire, {\em A note on obstructions and characteristic classes}, 
Amer.\ J.\ Math.\ {\bf 81} (1959), 773--784.

\bibitem{k60} 
M. A. Kervaire, {\em Some non-stable homotopy groups of Lie groups}, 
Illinois J.\ Math.\ {\bf 4} (1960), 161--169. 

\bibitem{k47} A.~Kirchoff, 
{\em Sur l'existence des certains champs tensoriels sur les sph\`eres  \`a $n$-dimensions},
C.\ R.\ Acad.\ Sci.\ Paris, {\bf 225}, (1947), 1258--1260.

\bibitem{t21}
D.~Kishimoto and I.~Membrillo-Solis and S.~Theriault, \emph{The homotopy types of {$\rm SO(4)$}-gauge groups}, Eur. J. Math. \textbf{7} (2021), no.~3, 1245--1252.

\bibitem{m63} 
J. Milnor, {\em Morse theory}, 
Based on lecture notes by M. Spivak and R. Wells,
Ann.\ Math.\ Stud.\ {\bf 51} Princeton University Press, Princeton, 1963

\bibitem{s53} H. Samelson,
{A connection between the Whitehead and the Pontryagin Product}, 
Amer.\ J.\ Math.\ {\bf 75} (1953), 744--752.

\bibitem{t67} E. Thomas,
{\em Complex structures on real vector bundles}, 
Amer.\ J.\ Math.\ {\bf 89} (1967), 887--908.

\bibitem{w62} C.~T.~C.~Wall, 
{\em Classification of $(n{-}1)$-connected $2n$-manifolds}, 
Ann.~of Math.~{\bf 75} (1962), 163--189.

\bibitem{w78} 
G.~W.~Whitehead, {\em Elements of homotopy theory},  Graduate Texts in Mathematics, Springer-Verlag 1978.

\bibitem{w07} C. Wockel,
{\em The Samelson product and rational homotopy for gauge groups},
Abh.\ Math.\ Sem.\ Univ.\ Hamburg {\bf 77} (2007), 219--228.

\end{thebibliography}
\end{document}